\documentclass[12pt]{amsart}
\usepackage{amsfonts}
\usepackage{amsmath}
\usepackage{amsthm}
\usepackage{bm,bbm}
\usepackage[dvips]{graphicx}
\usepackage{caption}
\usepackage{color}

\usepackage{bigints}

\numberwithin{equation}{section}
\allowdisplaybreaks

\theoremstyle{plain}
\newtheorem{theorem}{Theorem}[section]
\newtheorem{lemma}[theorem]{Lemma}

\newtheorem{corollary}[theorem]{Corollary}
\theoremstyle{definition}
\newtheorem{definition}[theorem]{Definition}
\newtheorem{remark}[theorem]{Remark}
\newtheorem{example}[theorem]{Example}

\def\beqn{\begin{equation}}
\def\beqn*{$$}
\def\eeqn{\end{equation}}

\newcommand{\bx}{{\bf x}}
\newcommand{\by}{{\bf y}}

\newcommand{\bi}{{\bf i}}
\newcommand{\bj}{{\bf j}}
\def\P{\mathbb{P}}
\def\F{\mathbb{F}}
\def\E{\mathbb{E}}
\def\Pn{\mathcal P_n}
\def\PH{\text{PH}}

\newcommand{\reals}{{\mathbb R}}
\newcommand{\bbr}{\reals}
\newcommand{\R}{\reals}
\newcommand{\bbn}{{\mathbb N}}

\newcommand{\X}{{\mathcal{X}}}
\newcommand{\Y}{{\mathcal{Y}}}
\newcommand{\I}{\mathcal I}
\newcommand{\Q}{\mathcal Q}
\newcommand{\B}{\mathcal B}
\newcommand{\Yp}{{\mathcal{Y}^{\prime}}}

\newcommand{\ta}{\theta}

\newcommand{\cechone}{\mathcal C_{1}}
\newcommand{\Rpn}{R}
\newcommand{\PHk}{\text{PH}_k}
\newcommand{\muk}{\mu^{(k)}}
\newcommand{\NnRkm}{\Phi_{n}^{(k,m)}}

\newcommand{\Nnkm}{\Phi_{n}^{(k,m)}}
\newcommand{\Nnkp}{\Phi_{n}^{(k,p)}}
\newcommand{\Nnkpone}{\Phi_n^{(k,p-1)}}
\newcommand{\NnRk}{\Phi_{n}^{(k)}}

\newcommand{\NtnRkm}{\widetilde \Phi_{n}^{(k,m)}}

\newcommand{\Ntnkp}{\widetilde \Phi_{n}^{(k,p)}}
\newcommand{\Nhnkp}{\widehat\Phi_n^{(k,p)}}
\newcommand{\NtnRk}{\widetilde \Phi_{n}^{(k)}}

\newcommand{\Nkp}{\Phi^{(k,p)}}
\newcommand{\BM}{\mathcal B_{2M}}
\newcommand{\pM}{Q_{2M}}
\newcommand{\matnp}{\begin{pmatrix} n \\ p \end{pmatrix}}

\newcommand{\pois}[1]{\mathrm{Poisson}\param{{#1}}}
\newcommand{\one}{{\mathbbm 1}}

\def\iid{\mathrm{iid}}
\newcommand{\param}[1]{\left(#1\right)}
\newcommand{\set}[1]{\left\{#1\right\}}
\def\cP{\mathcal{P}}
\def\cY{\mathcal{Y}}
\def\cC{\mathcal{C}}
\def\cR{\mathcal{R}}

\newcommand{\mean}[1] {\E\left\{{#1}\right\}}

\newcommand{\cech}{\v{C}ech }

\newcommand{\remove}[1]{}

\begin{document}

\bibliographystyle{abbrv}

\title[Convergence of Persistence Diagrams]
{Convergence of Persistence Diagrams for Topological Crackle}
\author{Takashi Owada}
\address{Department of Statistics\\
Purdue University \\
IN, 47907, USA}
\email{owada@purdue.edu}
\author{Omer Bobrowski}
\address{Department of Electrical Engineering\\
Technion--Israel Institute of Technology \\
Haifa, 32000, Israel}
\email{omer@ee.technion.ac.il}

\thanks{TO's research is partially supported by the NSF: Probability and Topology \#1811428.}


\begin{abstract}
In this paper we study the persistent homology associated with topological crackle generated by distributions with an unbounded support. Persistent homology is a topological and algebraic structure that tracks the creation and destruction of homological cycles (generalizations of loops or holes) in different dimensions. Topological crackle is a term that refers to homological cycles generated by ``noisy" samples where the support is unbounded.
We aim to establish weak convergence results for persistence diagrams -- a point process representation for persistent homology, where each homological cycle is represented by its $({birth,death})$ coordinates.
In this work we treat  persistence diagrams as random closed sets, so that the resulting weak convergence is defined in terms of the Fell topology. In this framework we show that the limiting persistence diagrams can be divided into two parts. The first part is a deterministic limit containing a densely-growing number of persistence pairs with a short lifespan. The second part is a two-dimensional Poisson process, representing persistence pairs with a longer lifespan. 
\end{abstract}

\maketitle

\section{Introduction} \label{sec:intro}
Persistent homology has emerged as a mathematical tool to analyze data in a way that is low-dimensional, coordinate-free, and robust to various deformations. The main idea is to extract topological ``features" from  data,   known as homological $k$-cycles (where $k$ represents dimension), in a multi-scale way that is stable under perturbations of the data.
Loosely speaking, a (nontrivial) $k$-cycle in a topological space is a structure that is topologically equivalent to a $k$-dimensional sphere (i.e.~the boundary of a $(k+1)$-dimensional ball). In order to find such structures in a dataset $\cP$, a common practice is to place balls of radius $r$ around the data, and consider their union $B_r(\cP)$.
Alternatively, one may construct a simplicial complex -- a higher dimensional notion of a graph that serves as a combinatorial representation for the geometric object. 
In this paper we will consider the \cech complex  generated by balls of radius $r$ around the sample, denoted $\cC_r(\cP)$ (see Section \ref{sec:geometric.complex} for a formal definition). 
\remove{In this paper we will consider the \cech and Vietoris-Rips complexes generated by balls of radius $r$ around the sample, denoted $\cC_r(\cP)$ and $\cR_r(\cP)$ respectively (see Section... for definitions).}

Considering the complex $\cC_r(\cP)$ and growing the parameter $r$, we get a nested sequence of complexes called filtration, in which $k$-cycles are created and destroyed (become trivial) at various times. Persistent homology is an algebraic structure that is designed to track these changes in cycles and produce a list of pairs $(birth,death)$ representing the time (radius) at which each cycle first appears in the filtration and the time at which it is terminated (becomes trivial, or ``gets filled"), respectively.

Commonly, the output of persistent homology (i.e.~a list of birth/death times) is summarized in a plot known as ``persistence diagram", see Figure \ref{fig:pd}.
In this plot, a single point is drawn for any $k$-cycle, for which the $x$-axis value represents its birth time, and the $y$-axis value represents the death time.

\begin{figure}
\centering
\includegraphics[scale=0.5]{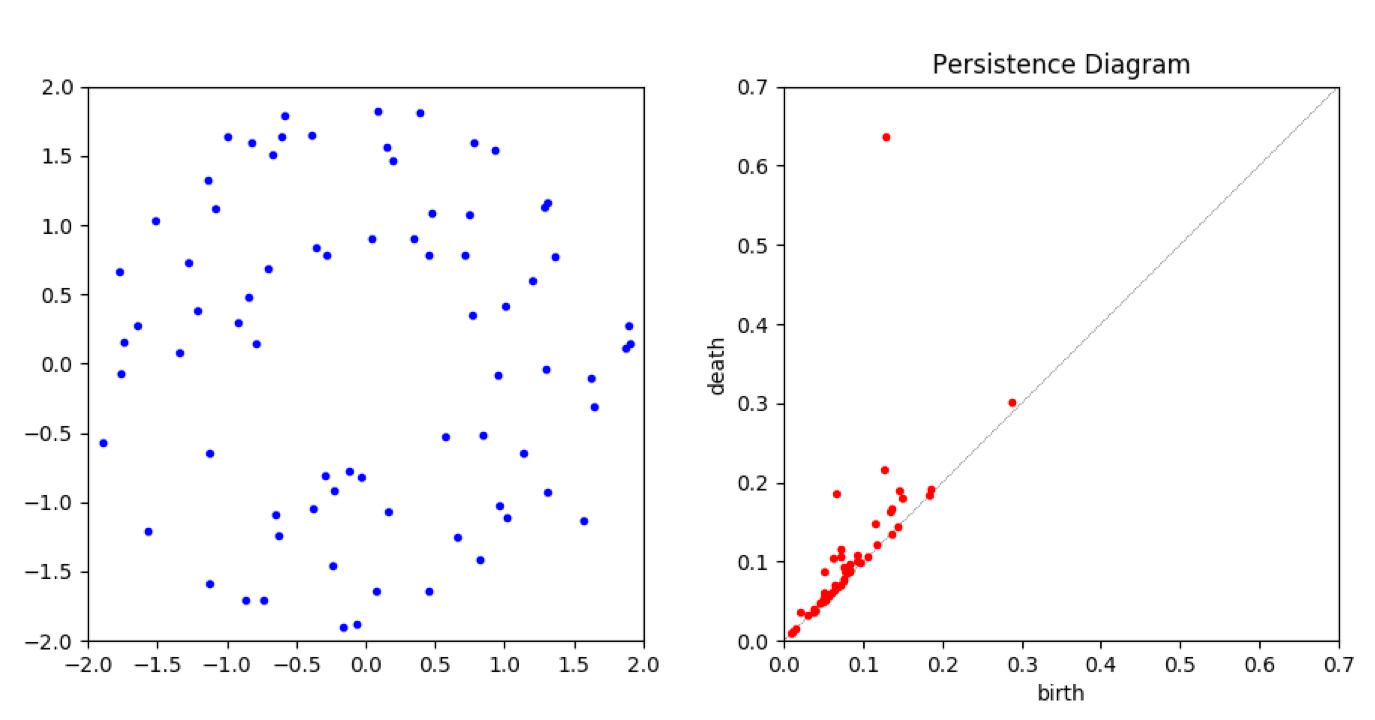}
\caption{\label{fig:pd} \footnotesize{The persistence diagram of a random \cech filtration. On the left figure the point process is generated on an annulus in $\R^2$. The persistence diagram on the right describes the birth and death times (radii) of all the $1$-cycles that appear in this filtration. Notice that most of the points in the persistence diagram are close to the diagonal where birth time equals death time, and one might consider these cycles as ``noise."
There is one point that stands out in the diagram, which corresponds to the hole of the annulus. The persistent homology was computed using the GUDHI library \cite{gudhi:urm}.}} 
\end{figure}

The study of homology and persistent homology generated by random data, started in \cite{kahle:2011}, and has been an active research topics over the past decade (see the survey in \cite{bobrowski:kahle:2018}).
Much of this study is dedicated to examining the behavior of noise (i.e.~point clouds that contain no intrinsic topological structure), and can be thought of as the development of ``null-models" for topological data analysis.

Most relevant to this paper is the work in \cite{duy:hiraoka:shirai:2018}, studying the distribution of points in the persistence diagram generated by point processes in a $d$-dimensional box. In \cite{duy:hiraoka:shirai:2018}, persistence diagrams are considered as Radon measures, for which the authors prove the existence of a limit in the form of a deterministic measure. In addution, they provide a law of large numbers and a central limit theorem for the persistent Betti numbers, i.e.~the number of $k$-cycles that exist over a given range of radii.

Quite commonly however, measurement noise may not be modelled by a distribution with a bounded support. In this case the analysis for persistence diagrams in \cite{duy:hiraoka:shirai:2018} is no longer valid. Previous works \cite{adler:bobrowski:weinberger:2014, owada:adler:2017} have studied fixed, rather than persistent, homology generated by distributions with unbounded supports. The first main finding was that if an underlying distribution generating data, has a tail at least as heavy as that of an exponential distribution, then nontrivial $k$-cycles keep appearing far away from the origin, 
regardless of the amount of points and  the choice of radius. The second  main finding was the emergence of a layered structure dividing the Euclidean space, with  each layer occupied by cycles of different degrees and  amounts. Briefly, as we get closer to the origin, the cycles of higher degrees appear and their number  increases. The fact that different regions in space are occupied by different types of structures at different quantities, suggests that  one may not look for a single limit theorem for fixed and persistent homology. Alternatively we could only provide separate limit theorems for each individual region. Similar phenomena have been pointed out in a series of works \cite{bobrowski:adler:2014,bobrowski:mukherjee:2015, duy:hiraoka:shirai:2018,kahle:meckes:2013,yogeshwaran:subag:adler:2017}, in which various limit theorems for topological invariants in different regimes were derived (though they are not directly related to the layered structure described above). 

The creation of an increasing number of cycles away from the origin has been given the name \textit{topological crackle}. The layered structure of the crackle, described in the last paragraph is visualized in Figure \ref{fig:crackle}. Topological crackle is typically generated by heavy tailed distributions, so the study of its topological features belongs to extreme value theory (EVT). EVT studies the extremal behavior (e.g., maxima) of stochastic processes with a variety of probabilistic and statistical applications. The standard literature on EVT includes \cite{resnick:1987,embrechts:kluppelberg:mikosch:1997,dehaan:ferreira:2006,resnick:2007}. In recent years many attempts have been made at understanding the geometric and topological features of multivariate extremes, among them \cite{balkema:embrechts:2007, balkema:embrechts:nolde:2010, schulte:thale:2012, decreusefond:schulte:thaele:2016} as well as \cite{adler:bobrowski:weinberger:2014, owada:adler:2017} cited above. 

\begin{figure}
\includegraphics[scale=0.36]{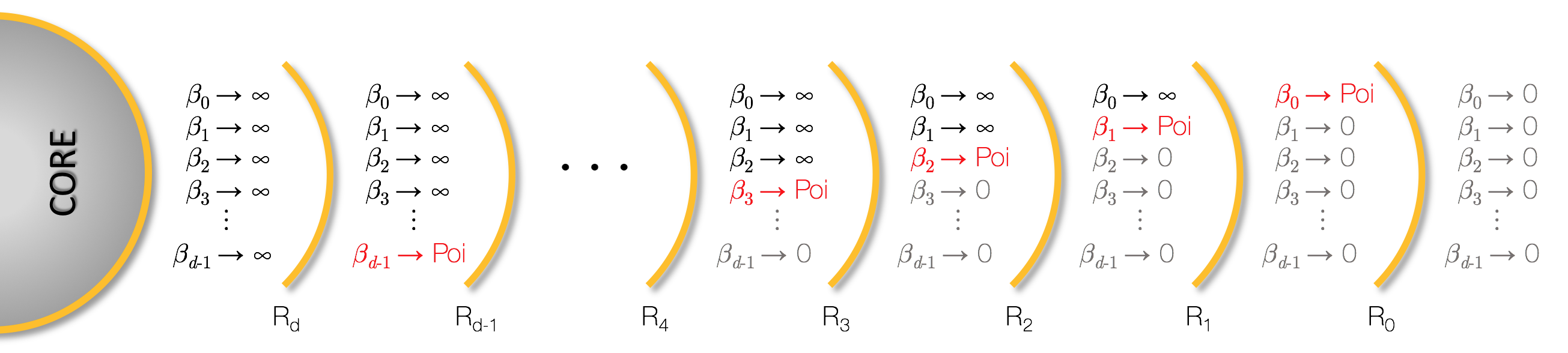}
\caption{\label{fig:crackle} \footnotesize{ Topological crackle is a layered structure of Betti numbers. ``Poi" stands for a Poisson distribution. For each individual layer (except most inner and outer ones), there is a unique $k \in \{  0,\dots,d-1\}$ such that the $k$th Betti number is approximated by a Poisson distribution, while all the other Betti numbers either vanish or diverge. }}
\end{figure}

In this paper we wish to study the probabilistic behavior of  persistence diagrams, under the assumption that the data are generated by a sequence of iid random variables, sampled from a distribution with a sufficiently heavy tail. Similarly to \cite{adler:bobrowski:weinberger:2014, owada:adler:2017}, we will divide $\R^d$ into layers of the form $L_R := \set{x\in \R^d : \|x\| \ge R}$, and establish limit theorems for the persistence diagram. As we will show later, for a given layer $L_R$, one can divide the persistence diagram into multiple regions, such that the number of persistence birth-death pairs grows at different orders of magnitude. This will be visualized later in Figure \ref{fig:pers_limit}. 
The main thrust of this paper is to study the limiting behavior of persistence diagram under the so-called Fell topology by treating persistence diagram as a random closed set in $\bbr^2$. A more formal discussion on the Fell topology is given in Section \ref{sec:Fell.topology}. Our approach is complementary to the relevant work \cite{duy:hiraoka:shirai:2018}, in which the authors treated persistence diagram as a random measure. 

The rest of this paper is organized as follows. In Section \ref{sec:preliminaries} we introduce the terminology used throughout the paper. Section \ref{sec:RV.tail} provides the main results of this paper, considering  heavy tailed distributions. In Section \ref{sec:exp.tail} we discuss the behavior of the model for exponentially decaying distributions. 
We will classify results in terms of heaviness of a tail of an underlying distribution. Such classification is typical in EVT. 
The proofs for both sections are presented in Section \ref{sec:proofs}. %Finally, we discuss conclusions and future work in Section \ref{sec:conclusion}.

\section{Preliminaries}  \label{sec:preliminaries}

\subsection{Geometric complexes}  \label{sec:geometric.complex}
An \emph{abstract simplicial complex} over a set $S$ is a collection of finite subsets $X \subset 2^S$ with the requirement that if $A\in X$ and $B\subset A$ then $B\in X$. 
A subset in $X$ of size  $k+1$ is called a \emph{$k$-simplex}, and commonly denoted as $\sigma = [x_0,\ldots,x_k]$.

In this work we discuss abstract simplicial complexes that are generated by a set of points $\cP\subset\R^d$, called \emph{geometric complexes}. 
Among many candidates of geometric complexes (see \cite{ghrist:2014}), the present paper focuses on one of the most studied ones, a \textit{\cech complex}. For construction we start by fixing a radius $r>0$, and drawing balls of radius $r$ around the points in $\cP$.
\remove{We will focus on two types of complexes, known as the \emph{\cech} and the \emph{Vietoris-Rips} complexes\footnote{Do we want to include Rips?}.} 
\begin{definition} \label{cech}
A \v{C}ech complex $\cC_r (\mathcal P)$ is defined by the following two conditions. 
\begin{enumerate}
	\item The 0-simplices are the points in $\mathcal{P}$.
	\item A $k$-simplex $[x_0, \dots, x_k]$ is in $\cC_r(\cP)$ if $\bigcap_{j = 0}^k B(x_{j}; r) \neq \emptyset$, 
\end{enumerate}
where $B(x; r) = \{y \in \R^d: |x - y| < r\}$ is an open ball of radius $r$ around $x \in \R^d$.  
\end{definition}
\remove{\begin{itemize}
\item {\bf The \cech complex:} \\
A subset $\set{p_1,\ldots,p_{k+1}}\subset \cP$ forms a simplex if $\cap_{i=1}^{k+1} B_r(p_i) \ne \emptyset$. The resulting complex is denoted $\cC_r(\cP)$.
\item {\bf The Vietoris-Rips complex:}\\
A subset $\set{p_1,\ldots,p_{k+1}}\subset \cP$ forms a simplex if $B_r(p_i)\cap B_r(p_j) \ne \emptyset$ for all $i,j$. The resulting complex is denoted $\cR_r(\cP)$.
\end{itemize}}

One of the key properties of the \cech complex $\cC_r(\cP)$, known as the \emph{Nerve Lemma} (see, e.g., Theorem 10.7 of \cite{bjorner:1995}), asserts that the union of balls $B_r(\cP) := \cup_{p\in \cP}B_r(p)$ and $\cC_r(\cP)$ are homotopy equivalent. In particular they have the same homology groups, i.e. for all $k \geq 0$, $H_k(B_r(\cP))\cong H_k(\cC_r(\cP))$.

\subsection{Persistent homology}
In this section we wish to describe homology and persistent homology in an intuitive and non-rigorous way, which is enough for the reader to follow the statements and proofs in this paper. We suggest \cite{carlsson:2009,ghrist:2008} as a good introductory reading, while a more rigorous coverage of algebraic topology is in \cite{hatcher:2002}. 

Let $X$ be a topological space. In this paper we will consider homology with field coefficients $\F$, in which case homology is essentially a sequence of vector spaces denoted $H_0(X),H_1(X), H_2(X),\ldots$. More specifically the basis of $H_0(X)$ corresponds to the connected components in $X$, and the basis of $H_1(X)$ corresponds to closed loops in $X$. The basis of $H_2(X)$ corresponds to cavities or ``air bubbles" in $X$, and generally, the basis of $H_k(X)$ corresponds to \emph{non-trivial $k$-cycles} in $X$. In addition to describing the topology of a single space $X$, homology theory also analyzes mappings between spaces. If $f:X\to Y$ is a map between topological spaces, then the induced map $f_*:H_k(X)\to H_k(Y)$ is a linear transformation, describing how $k$-cycles in $X$ are transformed into $k$-cycles in $Y$ (or disppear).

\emph{Persistent homology} can be thought of as a ``multi-scale" version of homology, designed to describe topological properties in a sequence of spaces. Let $\set{X_t}_t$ be a filtration of spaces, so that $X_s \subset X_t$ for all $s\le t$. In this case, one can consider the collection of vector spaces $\set{H_k(X_t)}_t$, together with the corresponding linear transformations $\imath_*^{(s,t)} : H_k(X_s)\to H_k(X_t)$ for all $s\le t$ induced by the inclusion map $\imath^{(s,t)}:X_s\to X_t$. Such a sequence is called a \emph{persistence module} (cf.~ \cite{carlsson:2009}).
Essentially, this sequence allows us to track the evolution of $k$-cycles as they are formed and terminated throughout the filtration. The theory developed for persistence modules allows for the definitions of \emph{barcodes}, which consist of intervals of the form $[\text{birth},\text{death})$, representing the 	time (the value of $t$) when a given cycle first appears and the time when it disappears, respectively.

Commonly, the information on the $k$th persistent homology is graphically provided via $k$th \emph{persistence diagram}. This is a two-dimensional plot, where each persistence interval of the form [birth, death) is represented as a single point, with the $x$-axis representing birth time and the $y$-axis representing death time. Figure \ref{fig:pd} shows an example of a persistence diagram generated by a \cech filtration $\set{\cC_r(\cP)}_r$, where $\cP$ is a random sample from an annulus.

For the study of the  filtration of geometric complexes, some structure must be imposed on persistence diagrams. First, notice that any persistence diagram is a subset of 
$$
\Delta:= \set{(x,y) : 0\le x \le y}, 
$$ 
as death times always come after birth times. 
%Further, for a finite subset of points $\cY\subset\cP$, it was shown in \cite{bobrowski:kahle:skraba:2017} that if $|\cY| = m$, the death/birth ratio for all intervals in the $k$th persistent homology generated by $\cY$, is bounded by $c m^{1/k}$ for some constant $c>0$. 
Further, given two positive integers $m$ and $k$, we fix an arbitrary birth time (radius) $b_0$. Then there exists an attainable maximum $d_0$ for the death time of $k$-cycles generated on $m$ points whose birth time is $b_0$ 
(which is $O(m^{1/k})$, see \cite{bobrowski:kahle:skraba:2017}).
Denoting $\pi_{k,m} = d_0/b_0$, the scaling invariance of persistent homology implies that all $k$-cycles generated on $m$ points are restricted to the region
\begin{equation*}
\Delta_{k,m} := \set{(x,y) : 0\le x \le y \le \pi_{k,m}x} \subset \Delta.
\end{equation*}
More precisely $\pi_{k,m}$ is a constant for which forming a persistence $k$-cycle is possible if $(x,y) \in \Delta_{k,m}$, but it becomes infeasible whenever $y > \pi_{k,m}x$. 
Notice that $m$ has to be at least $k+2$ in order to generate any cycles in the $k$th persistent homology for the \cech filtration. Finally, $\Delta_{k,m}$ is non-decreasing in $m$, that is, $\Delta_{k,m_1} \subset \Delta_{k,m_2}$ for all $m_1 \le m_2$; see Figure \ref{fig:pd_regions}.

\begin{figure}
\includegraphics[scale=0.25]{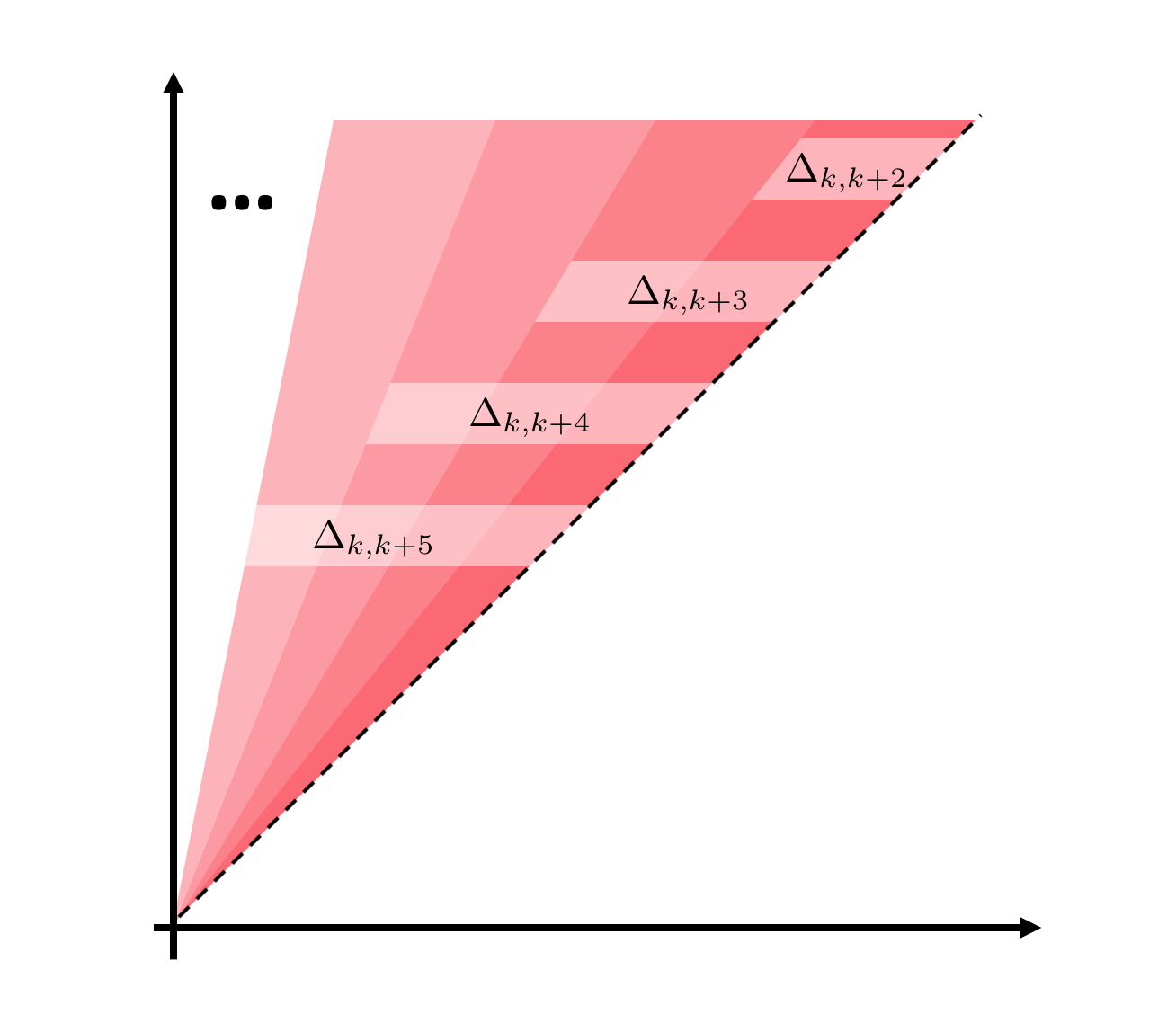}
\caption{\label{fig:pd_regions}
\footnotesize{ The structure of persistence diagrams generated by a \cech filtration.
Note that $\Delta_{k,m}$ represents the region in which the $k$th persistence pairs generated by subsets of size $m$, may appear. For a \cech filtration, $m$ has to be at least $k+2$.
}}
\end{figure}
\subsection{Fell topology}  \label{sec:Fell.topology}

The novel idea of the current paper is to treat random points in the persistence diagram as closed sets in $\Delta$. To this aim we introduce \textit{Fell topology}, which is perhaps the most standard topology on closed sets. Let $\mathcal F (\Delta)$ be the space of closed sets of $\Delta$. A sequence of closed sets $(F_n)$ converges to another closed set $F$ if and only if the following two conditions hold. 
\begin{itemize}
\item $F$ hits an open set $G$, i.e., $F\cap G \neq \emptyset$, implies there exists $N \geq 1$ such that for all $n\geq N$, $F_n$ hits $G$. 
\item $F$ misses a compact set $K$, i.e., $F \cap K =\emptyset$, implies there exists $N \geq 1$ such that for all $n\geq N$, $F_n$ misses $K$. 
\end{itemize}
By this property, Fell topology can be recognized as ``hit and miss" topology. 
 Stochastic properties of standard graphical tools used in EVT have been explored via convergence theorems under the Fell topology \cite{das:resnick:2008, ghosh:resnick:2010, das:ghosh:2013}. The main reference we use in this paper is \cite{molchanov:2005}.% and \cite{matheron:1975}.

The Fell topology is metrizable and hence induces a Borel $\sigma$-field $\B (\mathcal F (\Delta))$. Given a probability space $(\Omega, \mathcal A, \P)$, we say that $S: \Omega \to \mathcal F(\Delta)$ is a random closed set if 
$$
\{ \omega: S \cap K \neq \emptyset \} \in \mathcal A
$$
for every compact set $K$ in $\Delta$, i.e.~observing $S$, one can always determine if $S$ hits or misses any given compact set. Let us provide some facts about the  convergence of a sequence of random closed sets. 
Given random closed sets $(S_n)$ and $S$, the weak convergence $S_n \Rightarrow S$ in $\mathcal F(\Delta)$ is implied by 
$$
\P(S_n \cap K \neq \emptyset) \to \P(S\cap K \neq \emptyset)
$$
for every compact subset $K \subset \Delta$. For a measurable set $A \subset \Delta$ and $\epsilon >0$, denote by
\begin{equation}  \label{e:open.envelop}
(A)^{\epsilon- } = \bigl\{  (x,y)\in \Delta: d\bigl( (x,y), A \bigr)  < \epsilon \bigr\}
\end{equation}
an open $\epsilon$-envelop in terms of the Euclidean metric $d$. We say that $S_n$ converges to $S$ in probability if 
$$
\P \Big( \big[ \big( S_n \setminus (S)^{\epsilon - } \big) \cup \big( S \setminus (S_n)^{\epsilon -} \big)  \big]\cap K \neq \emptyset \Big) \to 0,  \ \ \ n\to\infty,
$$
for every $\epsilon > 0$ (see Definition 6.19 in \cite{molchanov:2005}). 

The primary benefit of our approach using the Fell topology is that one can establish limit theorems for the entire persistence diagram, even though the nature of the distribution of persistence birth-death pairs differs from region to region. To see this more clearly, let us consider the case where the persistence diagram is approximately divided into two regions, such that the persistence pairs are distributed densely in one region, and in the other region, the distribution is  much more sparse. This is roughly the same picture of persistence diagram as we will see in our main results. Then the previous works \cite{adler:bobrowski:weinberger:2014, owada:adler:2017, owada:2018} only established ``separate" limit theorems for each individual region. 
More specifically, the number of non-trivial cycles (i.e.~Betti number) obeys a central limit theorem if the cycles are distributed so densely that their  number grows to infinity as the sample size increases \cite{owada:2018}. On the other hand, if the spatial distribution of cycles is sparse enough, the Betti numbers will be governed by a Poisson limit theorem \cite{owada:adler:2017}. In contrast to these previous works, our approach allows to describe the entire persistence diagram by a ``single" limit theorem, which would help us get a whole picture of the limiting persistence diagram. 

The main discovery of the present paper is that the limiting persistence diagram as a random closed set splits into two parts. The first part is a deterministic subset of $\Delta$ of the form $\set{(x,y) : 0 \le x \le y \le c_1x}$ for some well-defined $c_1>0$. This part represents the  region containing a densely-growing number of persistence pairs. The second part is a two-dimensional Poisson process, supported on a region of the form $\set{(x,y) : 0 \le  x \le y \le c_2x}$, for some well defined $c_2>c_1$. Figure \ref{fig:pers_limit} provides a sketch of this behavior.

\section{Main results - Regularly Varying Tail Case} \label{sec:RV.tail}

In this section we describe in detail the problem studied in this paper, and present the main results. 
\subsection{Definitions}
The present section considers the following family of density functions with regularly varying tail. As is well-known in EVT, in the one-dimensional case, the regular variation of a tail completely characterizes the maximum-domain of attraction of a Fr\'echet distribution \cite{embrechts:kluppelberg:mikosch:1997}. 
\begin{definition}\label{def:density}
Let $f:\R^d\to\R$ be a probability density function. Let $S^{d-1}$ be the unit sphere in $\R^d$.
\begin{enumerate}
\item We say that $f$ is \emph{spherically symmetric} if $f(\rho \theta_1) = f(\rho \theta_2)$ for any $\rho\in \bbr_+$ and $\theta_1,\theta_2 \in S^{d-1}$.
For such functions we define $f(\rho) := f(\rho \theta)$ for any $\theta\in S^{d-1}$.
\item We say that a spherically symmetric $f$ has a \emph{regularly varying tail} if there exists $\alpha > d$ such that
\begin{equation}  \label{e:RV}
\lim_{\rho\to\infty} \frac{f(\rho t)}{f(\rho)} = t^{-\alpha} \ \ \text{for all  } t >0. 
\end{equation}
\end{enumerate}
\end{definition}

Let $X_1,X_2,\ldots$ be a sequence of $\iid$ random variables, having a common density function $f$ satisfying the conditions in Definition \ref{def:density}. Let $N_n \sim\pois{n}$ be a Poisson random variable, independent of $(X_i)$. Define the following point process
\begin{equation}\label{eq:pois}
	\cP_n := \begin{cases} \set{X_1,\ldots, X_{N_n}} & \text{if } \ N_n > 0,\\ 
	\emptyset & \text{if } \ N_n=0.\end{cases}
\end{equation}
Then one can show that $\cP_n$ is a spatial Poisson process on $\R^d$ with intensity function $nf$.

Let $R = R(n)$ be a sequence of $n$ growing to infinity, and consider 
$$
L_R = \big\{ x \in \bbr^d : \| x \| \geq R \big\} = \big( B(0; R) \big)^c,
$$
where $\| \cdot \|$ denotes Euclidean norm. 
The main objective in this paper is to study the ``extreme-value behavior" of the persistent homology 
for the \cech filtration. 
Namely we study the behavior of persistence cycles far away from the origin, generated by the points in $\mathcal P_{n,R} := \Pn \cap L_R$ for  large enough $n$. 
In other words, we aim to analyze the limiting distribution of persistent homology for the filtration $\{  \cC_{r} (\mathcal P_{n,R})\}_{r \ge 0}$. 
%Our approach will be to look at the corresponding persistence diagram as a point process in $\R^2$ and analyze its limiting distribution.

To that end, we define the following functions and objects. Recall that $\Delta = \bigl\{ (x,y)\in \bbr^2: 0\leq x \leq y \bigr\}$ is the upper infinite triangle in the first quadrant. Let $k \ge 1$ be an integer which remains fixed throughout the paper. 
For any finite set $\cY\subset \R^d$ let $\PHk (\Y)$ be the $k$th persistent homology generated by $\set{\cC_r(\cY)}_{r\ge 0}$. We now define the finite counting measure on $\Delta$,
\begin{equation}  \label{e:meas.mu}
\muk_{\Y} (\cdot) := \sum_{\gamma \in \PHk (\Y)} \delta_{(\gamma_b, \gamma_d)} (\cdot),
\end{equation}
where $\gamma$ represents a persistence interval in $\PH_k (\Y)$ whose endpoints $(\gamma_b,  \gamma_d)$ are the birth and death times (radii) respectively. Moreover, $\delta_{(x,y)}(\cdot)$ is the Dirac measure at $(x,y)$. In other words, $\mu_{\cY}^{(k)}$ represents all the pairs $(\gamma_b,\gamma_d)$ that appear in the $k$th persistence diagram generated by the set $\cY$. The finiteness of $\mu_{\cY}^{(k)}$  comes from a simple fact that if $|\Y|=m$, the number of $k$-cycles supported on $m$ vertices is bounded by the number of $k$-faces, which itself is bounded by $\binom{m}{k+1}$. 

We need a few more definitions before introducing the main point processes. For $m \ge 1$, define
$$
m(x_1,\dots,x_m) := \min_{1\leq i \leq m} \|  x_i\|,  \ \ x_i \in \bbr^d. 
$$
For a collection $\Y, \mathcal Z$ of points in $\bbr^d$ with $\Y \subset \mathcal Z$, 
\begin{equation}\label{eq:aux_func}
	\begin{split}
		h_R(\cY) &:= \one \set{m(\cY) \ge R},\\
		g_M(\cY,\mathcal Z) &:= \one\set{\cC_M(\cY) \text{ is a connected component of } \cC_M(\mathcal Z)},\\
		g_M(\cY) &:= g_M(\cY, \cY) = \one\set{\cC_M(\Y) \text{ is connected}}. 
	\end{split}
\end{equation}
The point processes we will examine in this paper are
\begin{equation}\label{eq:Phi}
\begin{split}
\NnRkm (\cdot) &:= \hspace{-5pt}\sum_{\Y\subset \Pn, \, |  \Y| = m}\hspace{-5pt} h_R(\cY) g_M(\cY,\cP_n) \muk_{\Y} (M\cdot\, ),\\
\NnRk (\cdot) &:= \sum_{m=k+2}^\infty \NnRkm (M\cdot),
\end{split}
\end{equation}
where $M=M(n)$ is a sequence of $n$, which will be explicitly determined  below together with $R=R(n)$. 
Note that $M \equiv$ constant is permissible as a special case. 
The process $\NnRkm$ represents the persistence $k$-cycles that are generated by the Poisson process $\cP_{n,R}$ defined above, such that the vertices forming these cycles belong to a single connected component of size $m$ in the complex $\cC_M (\cP_{n,R})$. 
By the construction of \eqref{eq:Phi} and the assumption that the random points generating data have a continuous distribution, the process $\Nnkm$ is simple (i.e.~$\sup_{x\in \Delta} \Nnkm (\{ x \}) \le 1$ a.s.) and finite. 
Forming a $k$-cycle in the \cech complex requires at least $k+2$ vertices; hence, the sum defining $\NnRk$ only starts at $m=k+2$. \remove{all possible $k$-cycles generated by $\cP_{n,R}$}Furthermore $\NnRk$ is almost surely a sum of finite number of point processes, as for $m>|\cP_n|$ we have $\NnRkm \equiv 0$.

\subsection{Weak convergence}

The primary goal in this paper is to prove a weak convergence theorem for $\NnRk$ as $n\to\infty$. 
Since the point process $\Nnkm$ in \eqref{eq:Phi} is simple and finite, the support of $\Nnkm$ is a finite random closed set (see Corollary 8.2 in \cite{molchanov:2005}). In this paper, by a slight abuse of notation, the letter $\Nnkm$ is used to denote both a point process and a random closed set as its support. 
In the latter treatment $\NnRk$ can be denoted as a union of $\Nnkm$'s,
$$
\NnRk = \bigcup_{m=k+2}^\infty \NnRkm. 
$$
This also represents a finite random closed set, because $\Nnkm = \emptyset$ whenever $m > |\Pn|$. 

Consequently, the topology we use for the weak convergence below is Fell topology (see Section \ref{sec:Fell.topology}) on closed sets of $\Delta$. 
 All the proofs, including those for corollaries that follow after Theorem \ref{t:main.fell}, are deferred to Section \ref{sec:proof.RV}. 
\begin{theorem}  \label{t:main.fell}
Let $f$ be a probability density function satisfying the conditions in Definition \ref{def:density}, and 
suppose that $R=R(n),M=M(n)$ are chosen such that $R\to \infty$, $M/R\to 0$ as $n\to\infty$. Furthermore, for some integer $p \ge k+2$, 
\begin{equation}  \label{e:asym.equ}
n^p M^{d(p-1)} R^d( f(R))^p \to 1, \ \ \ n\to\infty.
\end{equation}
Then
\begin{equation}  \label{e:main.fell}
\NnRk \Rightarrow \Nkp \cup B_{k,p-1} \ \ \text{in } \mathcal F (\Delta), \ \ \ n\to\infty, 
\end{equation}
where $\Rightarrow$ denotes weak convergence. The sets $\Nkp$ and $B_{k,p-1}$ are defined below.
\end{theorem}

The weak limit in Theorem \ref{t:main.fell} consists of two parts.

\begin{itemize}
\item 
$\Nkp$ - a (finite) random closed set characterized as a \textit{Poisson random measure} on $\Delta$, whose mean measure (intensity) is given by 
\begin{equation}  \label{e:mean.measure.heavy}
\E \big( |\Nkp \cap A| \big) := \frac{s_{d-1}}{p! (\alpha p -d)}\, \int_{(\bbr^d)^{p-1}} g_1(0,\by) \muk_{(0,\by)} (A) d\by, \ \ \ A \subset \Delta,
\end{equation}
where $| \cdot |$ denotes cardinality of a given set, 
$s_{d-1}$ is the volume of the $(d-1)$-dimensional unit sphere in $\R^d$, $\by = (y_1,\ldots, y_{p-1}) \in (\bbr^d)^{p-1}$, $(0,\by) = (0,y_1,\ldots, y_{p-1}) \in (\bbr^d)^p$, and so, $g_1(0,\by) = g_1(0,y_1,\ldots, y_{p-1})$.
\vspace{10pt}
\item $B_{k,p-1}$ - a non-random closed set of $\Delta$ defined as follows.
Recall that for a given subsets of size $m$, the $k$th persistence pairs $(\gamma_b, \gamma_d)$ are limited to the region  $\Delta_{k,m} = \set{(x,y) :  0 \le x \le y\le \pi_{k,m}x}\subset \Delta$. Next, define
\[
	b_{k,m} := \sup \set{\gamma_b : (\gamma_b,\gamma_d)\in \PHk(\cY),\  |\cY|=m, \ \cC_1(\cY)\text{ is connected}},
\]
i.e.~$b_{k,m}$ is the largest birth time for $k$-cycles that are generated on $m$ vertices and  connected at unit radius. Finally, define
\[
	B_{k,m} := \Delta_{k,m} \cap \big([0,b_{k,m}]\times \R_+\big).
\]
In other words, $B_{k,m}$ is the area in which the $k$th persistence pairs generated by subsets of $m$ points that are connected at unit radius, may appear; see Figure \ref{fig:pers_limit}. 
Note that $B_{k,m}$ is non-decreasing in $m$, i.e., $B_{k,m_1} \subset B_{k,m_2}$ for all $m_1 \le m_2$. 
\end{itemize}

\begin{figure}
\includegraphics[scale=.25]{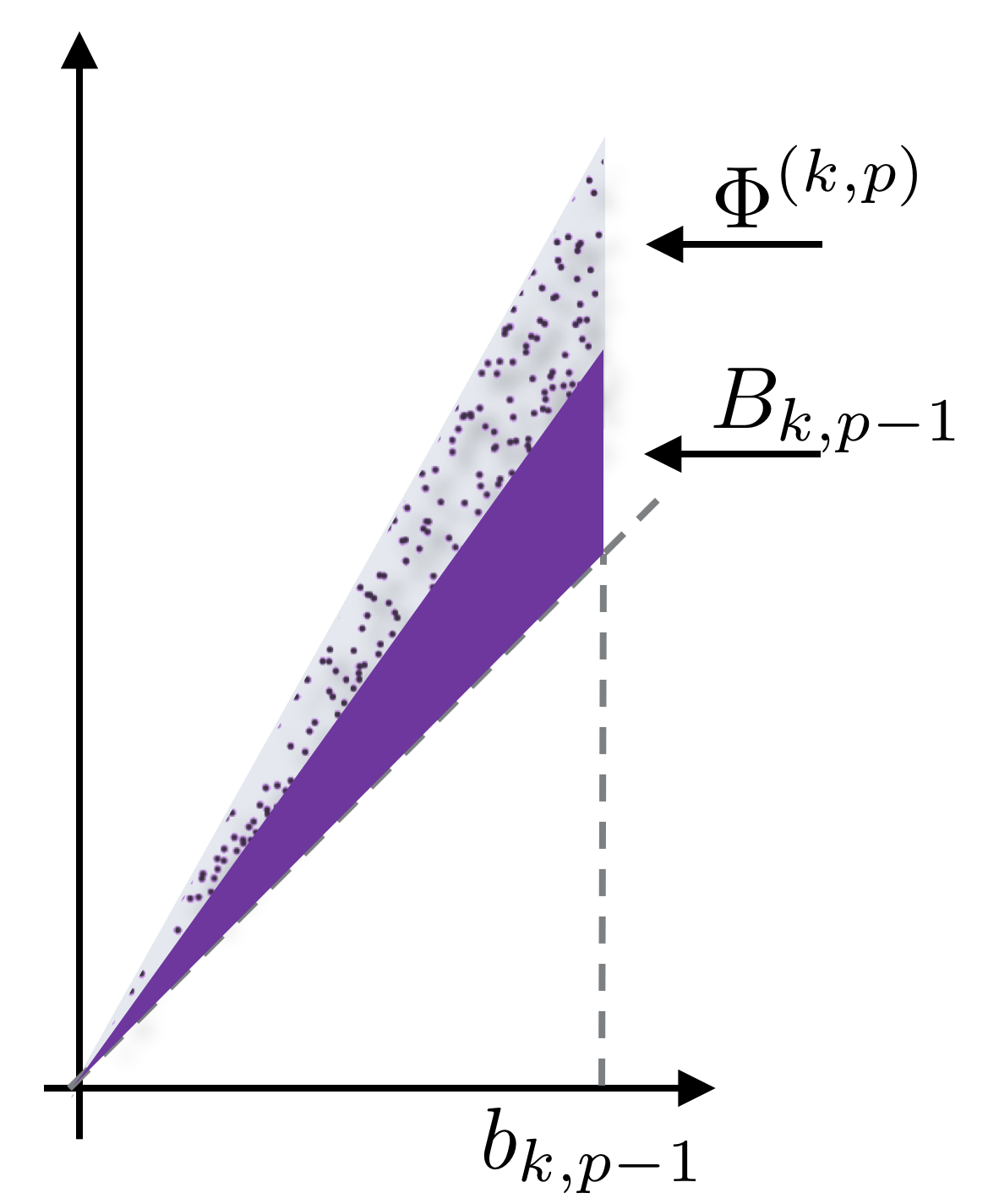}
\caption{\label{fig:pers_limit} \footnotesize{A cartoon for the limiting $k$th persistence diagram.
The solid area represents $B_{k,p-1}$, where the persistence pairs are dense enough, so that the entire region is covered. The grey area is $B_{k,p}\backslash B_{k,p-1}$, where only a finite number of Poisson points, denoted $\Phi^{(k,p)}$, are scattered in the limit. }}
\end{figure}

Let us provide some intuitions behind this theorem. As detailed in Lemmas \ref{l:expectation} and \ref{l:exp.variance}, we can show that for a measurable set $A \subset \Delta$ with $A \cap B_{k,k+2} \neq \emptyset$, 
\begin{equation}  \label{e:intuition1}
\E \big( | \Nnkp \cap A|  \big) \to \E \big( |\Nkp\cap A| \big) \in (0,\infty), \ \ \ n\to\infty,
\end{equation}
and as $n\to\infty$,
\begin{align}
\E \big( |\Nnkm \cap A | \big) &\to 0, \ \ \ \text{if } \ m >p,  \label{e:intuition2}  \\
\E \big(| \Nnkm \cap A | \big) &\to \infty, \ \ \ \text{if } \ m <p.   \label{e:intuition3}
\end{align}
Among these three results, the first indicates that there asymptotically exist at most finitely many persistence $k$-cycles that are generated on $p$ vertices in the complex $\cC_{M} (\cP_{n,R})$. Because of the rareness of persistence $k$-cycles, the set $\Nnkp$ will become ``Poissonian" in the limit. Additionally, \eqref{e:intuition2} implies that any $k$-cycles supported on more than $p$ vertices will vanish in the limit. In other words $\Nnkm$ converges to an empty set for all $m > p$. It thus follows that 
\begin{equation}  \label{e:quick.result1}
\bigcup_{m=p}^\infty \Nnkm \Rightarrow \Nkp \ \ \text{in }  \mathcal F(\Delta). 
\end{equation}
As for the remaining sets $\Nnkm$ for $m < p$, \eqref{e:intuition3} implies that there appear infinitely many persistence $k$-cycles as $n\to\infty$, that are generated on $m$ vertices in $\cC_{M}(\cP_{n,R})$. Accordingly, the union $\bigcup_{m=k+2}^{p-1} \Nnkm$ consists of infinitely many $k$th persistence pairs as $n\to\infty$, and ultimately, it converges to a deterministic closed set $\bigcup_{m=k+2}^{p-1}B_{k,m} = B_{k,p-1}$. Thus, 
\begin{equation}  \label{e:quick.result2}
\bigcup_{m=k+2}^{p-1} \Nnkm \Rightarrow B_{k,p-1} \ \ \text{in } \mathcal F(\Delta). 
\end{equation}
Finally combining \eqref{e:quick.result1} and \eqref{e:quick.result2} concludes \eqref{e:main.fell}. Section \ref{sec:proof.RV} gives a more formal argument. 

\begin{remark}
Note that \eqref{e:asym.equ} implicitly rules out a very quick decay of $M$. If $M$ decays to zero so quickly that 
$\limsup_{n\to\infty} n^pM^{d(p-1)} < \infty$, then \eqref{e:asym.equ} implies that 
$\liminf_{n\to\infty} \Rpn^d (f(\Rpn))^p>0$, but this contradicts with \eqref{e:RV}. At the same, the condition $M/R\to 0$ prevents a quick divergence of $M$. So the limiting behavior of $M$ is controlled on both sides.
\end{remark}
\begin{example}  \label{ex:RV.tail}
We consider a simple density with a Pareto tail, 
\begin{equation}  \label{e:Pareto.pdf}
f(x) = \frac{C}{1 + \|  x\|^\alpha}, 
\end{equation}
where $C$ is a normalizing constant. Taking $M \equiv 1$ and solving \eqref{e:asym.equ} with respect to $R$, we obtain
\begin{equation}  \label{e:RV.R}
R = (Cn)^{p/(\alpha p -d)}. 
\end{equation}
This sequence grows at a regularly varying rate with index $p/(\alpha p -d)$. Assuming \eqref{e:Pareto.pdf} together with other conditions in Theorem \ref{t:main.fell}, the weak convergence \eqref{e:main.fell} holds. 
\end{example}
%we consider a simple case for which 
%%\begin{equation}  \label{e:simple.pdf}
%\[
%f(x)  = \frac{C}{1+\| x \|^\alpha}\,, 
%\]
%%\end{equation}
%where $C$ is a normalizing constant, and $\| \cdot \|$ is the  Euclidean norm. In this case, the choice
%$$
%\Rpn  = C^{p/(\alpha p-d)}\bigl( n^p M^{d(p-1)} \bigr)^{1/(\alpha p-d)}. 
%$$
%satisfies \eqref{e:asym.equ}.
%In addition, we need to put an extra condition to preclude a quick divergence of $M_n$ as well: 
%\begin{equation}  \label{e:mild.cond}
%M_n/\Rpn \to 0, \ \ \ n\to\infty. 
%\end{equation}
%Now, combining \eqref{e:asym.equ} and \eqref{e:mild.cond} restricts ourselves to the case in which $M_n$ grows or decays at a ``moderate" speed. Note also that the case $M_n\equiv$ constant is permissible. In particular, if $f$ is given by \eqref{e:simple.pdf}, along with $M_n=n^q$, $-p/d(p-1) < q < 1/(\alpha-d)$, then we obtain 
%$$
%\Rpn = C^{p/(\alpha p -d)} n^{(p+dq(p-1)) / (\alpha p -d)}, 
%$$
%which grows at a regularly varying rate and satisfies \eqref{e:asym.equ} and \eqref{e:mild.cond}.

Before concluding this section, we state three corollaries of Theorem \ref{t:main.fell}. In the first corollary we assume that  instead of \eqref{e:asym.equ}, $R$ and $M$ satisfy
\begin{equation}  \label{e:cond.noPRM}
n^p M^{d(p-1)} R^d( f(R))^p \to 0, \ \ \ n^{p-1} M^{d(p-2)} R^d( f(R))^{p-1} \to \infty \ \ \text{as } n\to\infty. 
\end{equation}
To see the difference between \eqref{e:asym.equ} and \eqref{e:cond.noPRM}, we simplify the situation by assuming \eqref{e:Pareto.pdf} and taking $M\equiv 1$. It is then elementary to show that the $R$ satisfying \eqref{e:cond.noPRM} grows faster than the right hand side of \eqref{e:RV.R}, that is, $R^{-1} (Cn)^{p/(\alpha p-d)} \to 0$ as $n\to\infty$. This means that unlike \eqref{e:intuition1}, we obtain $\E \big( |\Nnkp \cap A| \big) \to 0$ as $n\to\infty$, in which case the random part $\Nkp$ vanishes from the limit.  A formal statement is given below. 

\begin{corollary}  \label{cor:noPRM}
Suppose that instead of \eqref{e:asym.equ}, $R$ and $M$ satisfy \eqref{e:cond.noPRM}. 
Then, 
$$
\NnRk \Rightarrow B_{k,p-1} \ \ \text{in } \mathcal F (\Delta).
$$
\end{corollary}
\remove{The intuition behind this corollary is the following. Denoting by $N_p$ the number of connected components in $\cC_{r,R}$ then we can show that $(\mean{N_p})^{-1} \sim n^p M^{d(p-1)} R^d( f(R))^p$.
If $R$ satisfies \eqref{e:asym.equ} then  $\mean{N_p}$ is finite, while $\mean{N_{m}} \to 0$ for $m>p$ and $\mean{N_{m}} \to \infty$ for $m<p$. The components of size $p$ constitute the $\Phi^{(k,p)}$ part in the limit (consisting a finite number of points), while the smaller components generate the $B_{k,p-1}$ (infinite) part.
Under the conditions of Corollary \ref{cor:noPRM}, we have that $\mean{N_p}\to 0$, therefore the finite part of the limit, i.e.~$\Phi^{(k,p)}$ vanishes.}

For the second corollary we again assume the condition at \eqref{e:asym.equ}.
We here aim to study the maximal lifespan (i.e.~death time -- birth time) of persistence $k$-cycles in the limiting persistence diagram. \remove{According to the basic philosophy in topological data analysis, the probabilistic features of a $k$-cycle with the longest (or at least sufficiently long) lifetime should be a central object to be analyzed. Indeed, such a $k$-cycle is believed to contain robust topological information, whereas the $k$-cycles of shorter lifetime are seen as mere topological noise. The standard machinery in extreme value theory using the continuous mapping theorem\footnote{ref?} enables us to use Theorem \ref{t:main.fell} to derive a limit theorem for the maxima as well as other marginals. Several examples, not necessarily relating to persistent homology, can be found in \cite{owada:adler:2016}. }For the required analyses, we need a continuous functional $T:\mathcal F(\Delta) \to \bbr_+$ defined by 
\begin{equation}  \label{e:vertical.distance}
T(F) = \sup_{(x,y)\in F} (y-x). 
\end{equation}
This functional captures the maximal vertical distance from the points in $F\subset \Delta$ to the diagonal line. 
For the remainder of this discussion, fix $t \in (0,b_{k,p-1})$, and define 
\[
\begin{split}
I_t &:=  \Delta \cap \bigl( [0,t] \times \bbr_+ \bigr),\\
J_t &:= \bigl\{ (x,y) \in  \Delta_{k,p} \cap I_t:y-x >T(\Delta_{k,p-1}\cap I_t) \bigr\}. 
\end{split}
\]
In other words, $J_t$ consists of points $(x,y)$ in $\Delta_{k,p} \cap I_t$ such that $y-x$ exceeds the maximal lifespan that can be attained by the points in $\Delta_{k,p-1}\cap I_t$. 
See Figure \ref{fig:maximal.lifespan}. 

\begin{figure}[!t]
\begin{center}
\includegraphics[scale=0.25]{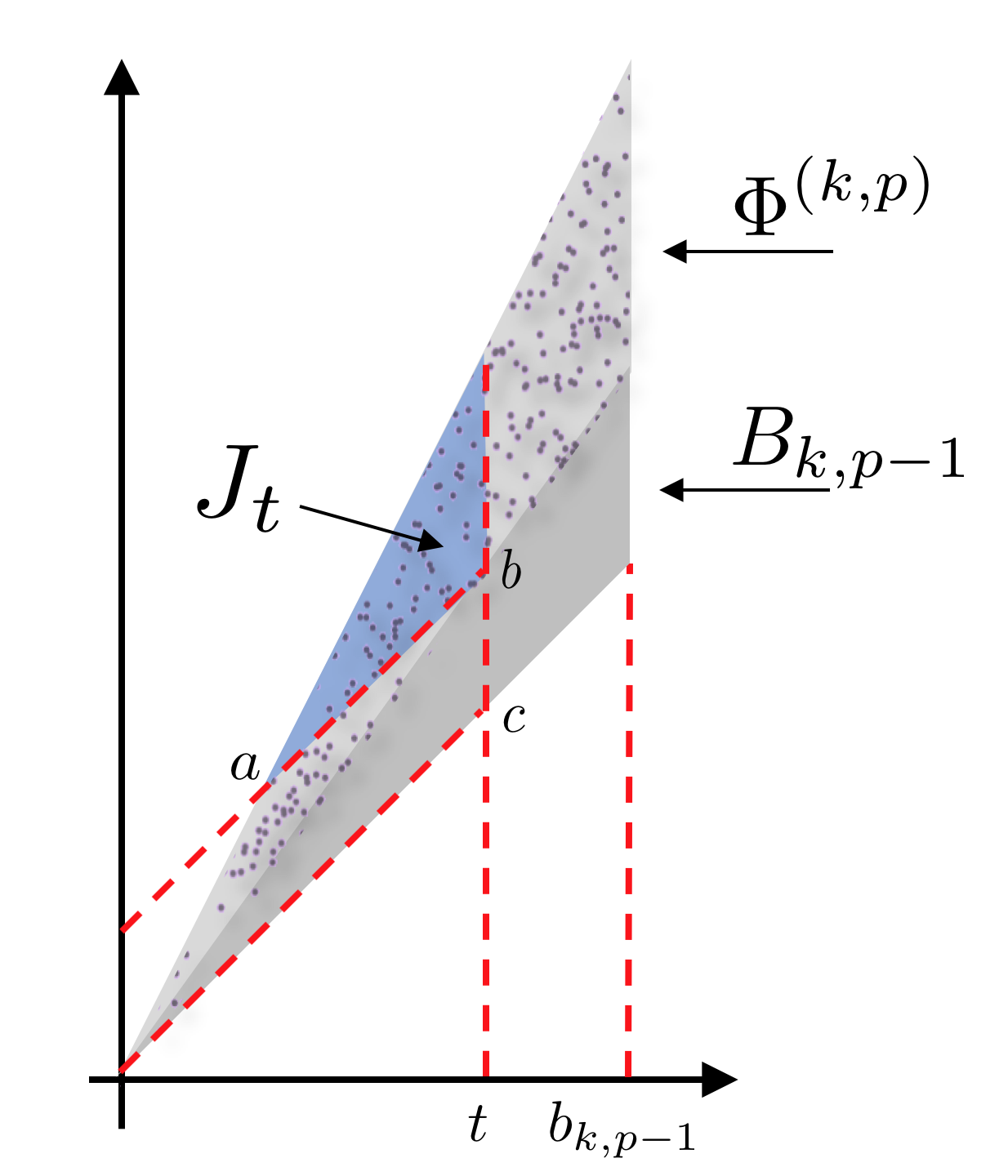}
\caption{{\footnotesize A dashed line segment $[a,b]$ is parallel to the diagonal line. The shaded area is $J_t$, in which there are at most finitely many Poisson points, representing persistence pairs with longer lifespan. The region $\Delta_{k,p-1} \cap I_t$ is densely covered by persistence pairs with shorter lifespan. If $J_t$ contains points, the maximal lifespan of persistence $k$-cycles can be attained by one of these points. If $J_t$ does not contain any points, the maximal lifespan is non-random and is equal to $T(\Delta_{k,p-1}\cap I_t)$, which is represented by a line segment $[b,c]$. }}
\label{fig:maximal.lifespan}
\end{center}
\end{figure}

The following corollary describes the limiting behavior of the $T(\NnRk \cap I_{Mt})$, i.e.~ the maximal lifespan of the persistence $k$-cycles generated by $\NnRk$, with the restriction that the birth time is less than $Mt$. The proof is immediate via continuous mapping theorem. Indeed applying a continuous functional \eqref{e:vertical.distance} to the weak convergence in Theorem \ref{t:main.fell} can yield the required result. 
%Now, adopting the notation as a random closed set, we consider  as a set of points in $\Delta$ generated by $m$-tuples  producing $k$-cycles before time $M_nt$, under the condition that each $m$-tuple is connected but is isolated from other points in $\Pn$ (with connectivity radius $2M_n$). The corollary below describes the limiting behavior of the longest lifetime of $k$-cycles generated by the points in 
%$$
%\NnRk \cap I_{M_nt} = \bigcup_{m=k+2}^\infty \bigl( \NnRkm \cap I_{M_nt} \bigr).
%$$

\begin{corollary}  \label{cor:maximal}
Under the assumptions of Theorem \ref{t:main.fell}, we have
$$
T\bigl( \NnRk \cap I_{Mt} \bigr) \Rightarrow Z_t \ \ \text{in } \bbr_+, \ \ \ n\to\infty,
$$
where 
$$
Z_t = \begin{cases}
\  T(\Delta_{k,p-1}\cap I_t) &  \text{if }\Nkp  \cap J_t = \emptyset, \\
 \ T\bigl( \Nkp \cap J_t \bigr) & \text{if } \Nkp \cap J_t \neq \emptyset. 
\end{cases}
$$
\end{corollary} 
The last statement  says that if the limiting Poisson random measure $\Nkp$ has no points in $J_t$, the weak limit $Z_t$ takes a purely deterministic value, and the ``non-random" set $\Delta_{k,p-1} \cap I_t$ yields the maximal lifespan. On the other hand, if $\Nkp$ has at least one points in $J_t$, the corresponding lifespan is necessarily longer than $T(\Delta_{k,p-1}\cap I_t)$. Then, the actual value of $Z_t$ is random. 
%\footnote{This entire section sounds very obvious, unless I'm missing something...}
%\begin{proof}
%By virtue of Theorem \ref{t:main.fell}, 
%$$
%\NnRk \cap I_{M_nt} \Rightarrow \bigl( \Nkp \cup B_{p-1} \bigr) \cap I_t \ \ \text{in } \mathcal F(\Delta), \ \ \ n\to\infty. 
%$$
%By the continuous mapping theorem, 
%$$
%T\bigl( \NnRk \cap I_{M_nt}  \bigr) \Rightarrow T\bigl( (\Nkp \cup B_{p-1}) \cap I_t \bigr) \ \ \text{in } \bbr_+, \ \ \ n\to\infty. 
%$$
%To identify the limiting distribution, we have, for $x\geq0$, 
%\begin{align*}
%\P \Bigl(  T\bigl(( \Nkp \cup B_{p-1}) \cap I_t \bigr) \leq x  \Bigr) &= \P \bigl( \Nkp \cap J_t=\emptyset \bigr)\, \one \bigl\{ T(B_{p-1}\cap I_t) \leq x \bigr\} \\
%&+ \P \Bigl(  \Nkp \cap J_t \neq \emptyset, \, T \bigl( \Nkp \cap J_t \bigr) \leq x \Bigr); 
%\end{align*}
%therefore, 
%$$
%T\bigl( (\Nkp \cup B_{p-1}) \cap I_t \bigr) \stackrel{d}{=} Z_t
%$$
%as desired. 
%\end{proof} 

For the third corollary, recall that $\bigcup_{m=p}^\infty \Nnkm$ asymptotically consists of the $k$th persistence pairs that are generated by a finite number of components of size $p$. Since the number of persistence pairs is always finite, the weak convergence can be reformulated as that in the space $\text{MP}(\Delta)$ of locally finite counting measures on $\Delta$. We here equip $\text{MP}(\Delta)$ with the vague topology (see \cite{resnick:1987}).
\begin{corollary}  \label{cor:vague}
Under the conditions of Theorem \ref{t:main.fell}, we take $\NnRkm$ as a point process. Then  
$$
\sum_{m=p}^\infty \NnRkm \Rightarrow \Nkp \ \ \text{in } \text{MP}(\Delta).
$$
\end{corollary}

\bigskip

\section{Exponentially decaying tails case}  \label{sec:exp.tail}

In the present section we wish to study the case where the distribution generating random points has an exponentially decaying tail.
The results in this case are  parallel to those of the previous section except for the normalization and limiting distributions. 

To define the density function, we use the \textit{von-Mises function}. The following setup is somewhat typical in EVT; see \cite{balkema:embrechts:2007,balkema:embrechts:nolde:2010,owada:adler:2017,owada:2017}. 

\begin{definition}[von-Mises function]\label{def:von_mises}
We say that $\psi:\R_+\to \R$ is a \emph{von-Mises function}, if $\psi$ is $C^2$, $\psi'(z) >0$, and 
\[
	\lim_{z\to\infty} \psi(z) = \infty,\qquad \lim_{z\to\infty}\frac{d}{dz}\left(\frac{1}{\psi'(z)}\right) = 0.
\]
\end{definition}
In this section we study density functions $f:\R^d\to\R_+$ of the form
\begin{equation}\label{eq:f_vm}
f(x)  = L(\| x \|) e^{-\psi(\| x\|)}. 
\end{equation}
Let $a(z):= 1/\psi^\prime(z)$, then from Definition \ref{def:von_mises} we have 
that $a^\prime(z) \to 0$, $z\to\infty$. Therefore,
\begin{equation}  \label{e:cesaro}
\lim_{z\to\infty}\frac{a(z)}{z} = 0.
\end{equation}
We assume that $L:\bbr_+ \to \bbr_+$ is \textit{flat} for $a$, that is,
\begin{equation}  \label{e:flat}
\lim_{t\to\infty}\frac{L\bigl( t+a(t)v \bigr)}{L(t)}= 1, 
\end{equation}
uniformly on $v \in [-K,K]$ for every $K>0$. 
Furthermore we assume that for some $\gamma\geq0$, $z_0 > 0$ and $C\geq 1$, we have
\begin{equation}  \label{e:L.gamma}
\frac{L(zt)}{L(z)} \leq Ct^\gamma \ \ \text{for all } t>1, z\geq z_0. 
\end{equation}

Condition \eqref{e:flat} together with \eqref{eq:f_vm}  implies that the tail of $f$ is determined by the function $\psi$ (or equivalently $a$), and is independent of $L$. Thus, we can classify $f$ in terms of the asymptotics of $a$. If $a(z)$ converges to a positive, finite constant as $z\to\infty$, we say that $f$ has an (asymptotic) exponential tail. If $a(z)$ diverges as $z\to\infty$ we say that $f$ has a \textit{subexponential} tail, and finally, if $a(z) \to0$, we say that $f$ has a \textit{superexponential} tail.

%
%In this case we will require $(R, M)$ to satisify
%\begin{equation}  \label{e:asym.equ2}
%n^p M^{d(p-1)} a(\Rpn) \Rpn^{d-1} f(\Rpn e_1)^p \to 1, \ \ \ n\to\infty. 
%\end{equation}
As in the previous section, we need to choose a radius $R=R(n)$ and connectivity value $M=M(n)$, for topological crackle to occur.  In \cite{owada:adler:2017} it was shown that the occurrence of topological crackle  depends on the limit
$$
c := \lim_{n\to\infty} \frac{a(R)}{M}.
$$
In particular, if $c=0$,
 topological crackle never occurs, and random points are densely scattered near the origin, so that placing unit balls around the points constitutes a topologically contractible object called \textit{core}; see  \cite{adler:bobrowski:weinberger:2014}.  Since the main focus of the present work is topological crackle, we do not treat the case $c=0$ and always assume  $c \in (0,\infty]$. 
By definition, if $M$ is a positive constant and $c\in (0,\infty]$, then $f$ never has a superexponential tail. 

We now describe  a series of results analogous to those in the previous sections. The proof is presented in Section \ref{sec:proof.exp.tail}. 

\begin{theorem}  \label{t:main.fell.exponential}
Let $f$ be a probability density function of the form \eqref{eq:f_vm}, and 
suppose that $R=R(n),M=M(n)$ are chosen such that $R\to \infty$, $M/R\to 0$, $a(R)/M\to c\in(0,\infty]$ as $n\to\infty$. Moreover for some integer $p \ge k+2$,
\begin{equation}  \label{e:asym.equ.exponential}
n^p M^{d(p-1)} a(\Rpn) \Rpn^{d-1}( f(R))^p \to 1, \ \ \ n\to\infty.
\end{equation}
Then,
$$
\NnRk \Rightarrow \Nkp \cup B_{k,p-1}\ \ \text{in } \mathcal F (\Delta),
$$
where $\Nkp$ is defined below, and $B_{k,p-1}$ is the same non-random set as in Theorem \ref{t:main.fell}. 
\end{theorem}
Similarly to Theorem \ref{t:main.fell}, the limit $\Nkp$ above is a (finite) random closed set characterized as a 
Poisson random measure. Here, the mean measure of $\Nkp$ is given by
\begin{align}
\frac{1}{p!}\, \int_0^\infty d\rho &\int_{S^{d-1}} \hspace{-5pt}J(\ta)d\ta \int_{(\bbr^d)^{p-1}}\hspace{-5pt}d\by\, g_1(0,\by)\, \muk_{(0,\by)} (\cdot) \label{e:mean.meas.light} \\
& \times e^{-p\rho - c^{-1} \sum_{i=1}^{p-1} \langle \ta, y_i \rangle}\, \one \bigl\{ \rho + c^{-1} \langle \ta, y_i \rangle \geq 0, \ i=1,\dots,p-1 \bigr\},  \notag
\end{align}
where $\langle \cdot, \cdot \rangle$ denotes scalar product and
\begin{equation}  \label{e:Jacobian}
J(\theta) = \sin^{d-2}(\theta_1) \sin^{d-3} (\theta_2) \cdots \sin (\theta_{d-2})
\end{equation} 
is the Jacobian. 
Interestingly, if $c=\infty$, \eqref{e:mean.meas.light} coincides with \eqref{e:mean.measure.heavy} up to multiplicative constants, implying that the two limiting Poisson random measures coincide regardless of heaviness of the tail of an underlying distribution. 

Notice that the main difference between \eqref{e:asym.equ} and \eqref{e:asym.equ.exponential} lies only in the growth rate of $\Rpn$. To see this, take $M\equiv 1$, and consider the simple example 
$$
f(x) = Ce^{-\| x \|^\tau/\tau}, \quad 0 < \tau \leq 1.
$$
Then $a(z) = z^{1-\tau}$,  and the solution to \eqref{e:asym.equ.exponential} is given by 
$$
\Rpn = \bigl( \tau \log n + p^{-1}(d-\tau) \log (\tau \log n) + \tau \log C \bigr)^{1/\tau},
$$
which grows logarithmically, whereas, as seen in Example \ref{ex:RV.tail}, the $\Rpn$ in the heavy tail setup grows at a regularly varying rate. 

%{\color{red}
%Furthermore, we shall simply take over condition \eqref{e:mild.cond} of the previous section, i.e., 
%$$
%M_n / \Rpn \to 0 \ \ \text{as } n\to\infty,
%$$
%to preclude a very quick divergence of $M_n$ relative to $\Rpn$.
%}

We now present the statements equivalent to those in Corollaries \ref{cor:noPRM}, \ref{cor:vague}, and \ref{cor:maximal}. 
\begin{corollary} 
Suppose that instead of \eqref{e:asym.equ.exponential}, $R$ and $M$ satisfy 
\begin{align*}
n^p M^{d(p-1)} a(R) R^{d-1} (f(R))^p &\to 0, \ \ \ n\to\infty,\\
n^{p-1} M^{d(p-2)} a(R) R^{d-1} (f(R))^{p-1} &\to \infty, \ \ \ n\to\infty. 
\end{align*}
Then
$$
\NnRk  \Rightarrow B_{k,p-1} \ \ \text{in } \mathcal F(\Delta). 
$$
\end{corollary}
\begin{corollary}
Under the assumptions of Theorem \ref{t:main.fell.exponential}, we have, as $n\to\infty$, 
$$
T\bigl( \NnRk \cap I_{Mt} \bigr) \Rightarrow Z_t \ \ \text{in } \bbr_+, \ \ \ n\to\infty,
$$
where $Z_t$ is given by 
$$
Z_t = \begin{cases}
\ T\big( \Delta_{k,p-1}\cap I_t \big)& \text{if } \Nkp \cap J_t = \emptyset, \\
 \ T\bigl( \Nkp \cap J_t \bigr) & \text{if } \Nkp \cap J_t \neq \emptyset. 
\end{cases}
$$
\end{corollary} 
\begin{corollary} 
Under the assumptions of Theorem \ref{t:main.fell.exponential}, we take $\NnRkm$ as a point process. Then
$$
\sum_{m=p}^\infty \NnRkm \Rightarrow \Nkp \ \ \text{in } \text{MP}(\Delta). 
$$
\end{corollary}
\bigskip

\section{Proofs}  \label{sec:proofs}

In this section we provide the proofs for all the statements in this paper.
We split the proofs between the regularly varying and the exponentially decaying tail cases.

\subsection{Some notation} 
The following notation will be used throughout the proofs. For $x\in \bbr^d$, $\by =(y_1,\dots,y_{m}) \in (\bbr^d)^{m}$, and $r>0$, define
$$
x+r\by := (x+ry_1, \cdots, x+ry_{m}) \in (\bbr^d)^m. 
$$
\remove{Let $f:\R^d\to \R$ be a probability density function. 
For $\bx=(x_1,\dots,x_m) \in (\bbr^d)^m$,  define
$$
f(\bx) := f(x_1) \cdots f(x_m),
$$
i.e.~the joint density of $m$ iid variables.
Further, for $x\in \bbr^d$, $\by =(y_1,\dots,y_{m-1}) \in (\bbr^d)^{m-1}$, and $r>0$, define
$$
f(x+r\by) := f(x+ry_1) \cdots f(x+ry_{m-1}). 
$$
The same notation will similarly apply to other functions and objects.}
The proofs will involve calculating certain volumes, which we define next.
Let
$$
\mathcal B_r(\bx) := \bigcup_{i=1}^m B(x_i; r), \ \ \bx = (x_1,\dots,x_m) \in (\bbr^d)^m,
$$
be a union of $m$ closed balls of radius $r$, and let
\begin{equation}\label{eq:Q}
Q_r(\bx) := \int_{\mathcal B_r (\bx)} f(z) dz,
\end{equation}
be the probability measure of the given union of balls.

We denote by $\lambda_m$ the Lebesgue measure on $\bbr^m$.
Finally, the notation $C^*$ will represent a generic positive constant, which does not depend on $n$ and may vary between (or even within) the lines. 
%Finally we provide a basic fact: for every $m\geq 1$, 
%\begin{equation}  \label{e:basic.fact}
%\matnm \sim \frac{n^m}{m!} \ \ \text{as } n\to\infty,
%\end{equation}
%where $\sim$ means that the ratio of the two sides tends to $1$ as $n\to\infty$. 

\subsection{Regularly varying tails}  \label{sec:proof.RV}
Our main goal in this section is to prove the results for the regularly varying tail case. We do not present the proof of Corollary \ref{cor:noPRM}, since it is very similar to that for Theorem \ref{t:main.fell}. Further, the proof of Corollary \ref{cor:maximal}  will be skipped, because the statement is nearly obvious. 
We will use the following auxiliary point process.
Recalling the definitions of a counting measure at \eqref{e:meas.mu} and $h_R, g_M$ in \eqref{eq:aux_func}, we define
 
%In the below we mainly present the proof of Theorem \ref{t:main.fell}. Indeed, Corollary \ref{cor:noPRM} can be proved in totally the same (or even easier) manner, and Corollary \ref{cor:vague} is a direct consequence of Theorem \ref{t:main.fell}. For the proof, we start wtih presenting two lemmas for which we need to introduce extra point processes as a variant of \eqref{e:ppc} and \eqref{e:ppc.sum}:
\begin{equation}  \label{e:ppc.tilde}
\begin{split}
\NtnRkm (\cdot) &:=\hspace{-5pt} \sum_{\Y\subset \Pn, \, |  \Y| = m}\hspace{-5pt}h_R(\cY) g_M(\cY) \muk_\Y (M \, \cdot\, )\\
\NtnRk (\cdot) &:= \sum_{m=k+2}^\infty \NtnRkm (\cdot).
\end{split}
\end{equation}
The only difference between $\NnRkm$ and $\NtnRkm$ is that the latter does not require the subsets $\cY$ to form a connected component of $\cC_{M}(\cP_{n,R})$, i.e.~$\Y$ does not need to be isolated from the rest of the complex. Consequently, we have
$
\NtnRkm (\cdot) \geq \NnRkm (\cdot)
$.
As in the case of \eqref{eq:Phi}, we may and will denote by \eqref{e:ppc.tilde} the corresponding random closed sets. Indeed the proof below uses \eqref{eq:Phi} and \eqref{e:ppc.tilde} as random closed sets only, except for the argument for Corollary \ref{cor:vague}. 

%Depending on the context, however, we may treat \eqref{e:ppc.tilde} and \eqref{e:ppc.tilde.sum} as random closed sets.
%Using these point processes (or random closed sets), the first lemma below describes the asymptotic expectations of $\NnRkp(A)$ and $\NtnRkp (A)$, both of which frequently apply for  proving \eqref{e:prm.part}. The second lemma gives information on the asymptotic moments of $N_{n,\Rpn}^{(k,p-1)} (A)$, which shall be utilized for the proof of \eqref{e:closed.set.part}. 

%Before commencing the main body of the proof, we want to introduce several shorthand notations to save spaces. 
%
We start with two lemmas to evaluate certain asymptotic moments. 

\begin{lemma} \label{l:expectation}
Let $A\subset \Delta$ be a measurable set, such that $A \cap B_{k,p} \neq \emptyset$. Under the assumptions of Theorem \ref{t:main.fell}, 
\[
\lim_{n\to\infty}\E \bigl( |\Nnkp \cap A| \bigr) = \lim_{n\to\infty}\E \bigl( |\Ntnkp \cap A| \bigr) = \E \bigl( |\Nkp \cap A| \bigr) \in (0,\infty).
\]
\end{lemma}

\begin{lemma} \label{l:exp.variance}
Let $A\subset \Delta$ be a measurable set, with $A\cap B_{k,p-1}\neq \emptyset$. Under the assumptions of Theorem \ref{t:main.fell}, 
\[
\begin{split}
\E\bigl( | \Nnkpone \cap A| \bigr) &\sim C_1 \bigl( nM^df(\Rpn) \bigr)^{-1}, \ \ \ n\to\infty, \text{ and } \\
\text{Var} \bigl(| \Nnkpone \cap A| \bigr) &\leq C_2 \bigl( nM^df(\Rpn ) \bigr)^{-1},
\end{split}
\]
for some $C_1, C_2>0$, which are independent of $n$. 
%where
%$$
% C_1 =  \frac{s_{d-1}}{(p-1)! \bigl( \alpha (p-1)-d \bigr)}\, \int_{(\bbr^d)^{p-2}} \one \bigl\{  \cechone (0,\by) \text{ is connected} \, \bigr\}\, \muk_{(0,\by)}(A)d\by 
%$$
%is a positive constant. Note that \eqref{e:asym.equ} implies the expectation diverges as $n\to
%\infty$. 

%Furthermore, there exists a constant $C_2>0$ such that 
%\begin{equation}  \label{e:exp.variance2}
%\end{equation}
\end{lemma}

\begin{proof}[Proof of Lemma \ref{l:expectation}]
We will prove the limit for $\Nnkp$ only, since the limit for $ \Ntnkp$ can be proved in the same way.
It follows from the Palm theory for Poisson processes (e.g., Section 1.7 in \cite{penrose:2003}) that 
\begin{equation}  \label{e:palm.theory.first.lemma}
\E \bigl(| \Nnkp \cap A| \bigr)  = \frac{n^p}{p!}\, \E \Bigl[ \, h_R(\X_p) g_M(\X_p,\X_p \cup \cP_{n})\muk_{\X_p} (MA) \Bigr],
\end{equation}
where $\X_p=(X_1,\dots,X_p)$ is a set of $p$ iid points with probability density $f$, and independently of $\Pn$. Note that for the set $\X_p$ to be disconnected from the rest of the complex $\cC_M(\cP_{n,R})$, we require that $\cP_n\cap\BM(\X_p)=\emptyset$.
Therefore, by the conditioning on $\X_p$ we have
\begin{align*}
\E \bigl(| \Nnkp \cap A| \bigr)  &= \frac{n^p}{p!}\, \E \Bigl[h_R(\X_p) g_M(\X_p) 
  \muk_{\X_p} (MA) \P \bigl( \Pn \cap \BM (\X_p) = \emptyset \bigl| \X_p \bigr) \Bigr] \\
&= \frac{n^p}{p!}\, \E \Bigl[  h_R(\X_p) g_M(\X_p) 
  \muk_{\X_p} (MA) e^{-n\pM (\X_p)} \Bigr] \\
&= \frac{n^p}{p!}\, \int_{(\bbr^d)^p} h_R(\bx) g_M(\bx) \mu_{\bx}^{(k)}(MA)e^{-n\pM (\bx)} \prod_{i=1}^p f(x_i) d\bx.
\end{align*}
Performing the change of variables $x_1\leftrightarrow x$, $x_i \leftrightarrow x+M y_{i-1}$, $i=2,\dots,p$, we have
\begin{align*}
\E \bigl( |\Nnkp \cap A| \bigr)  &= \frac{n^p}{p!}\, M^{d(p-1)} \int_{\bbr^d}dx \int_{(\bbr^d)^{p-1}} d\by h_R(x,x+M\by) g_M(x, x+M\by)   \mu_{(x,x+M\by)}^{(k)}(MA)\\
&\qquad \qquad \qquad \qquad \qquad \qquad  \times e^{-n \pM(x,x+M\by)}f(x) \prod_{i=1}^{p-1} f(x+My_i) \\
&= \frac{n^p}{p!}\, M^{d(p-1)} \int_{\bbr^d}dx \int_{(\bbr^d)^{p-1}} d\by  h_R(x,x+M\by) g_1(0,\by) \mu_{(0,\by)}^{(k)}(A)\\
&\qquad \qquad \qquad \qquad \qquad \qquad  \times e^{-n \pM(x,x+M\by)}f(x) \prod_{i=1}^{p-1} f(x+My_i),
\end{align*}
where the second equality follows from the translation invariance and scaling properties of $g_M$ and $\mu^{(k)}$. 
Next, we apply a polar coordinate transform $x\leftrightarrow (r,\ta)$ where $r\in [0,\infty)$ and $\ta\in S^{d-1}$, which is followed by another change of variable $r\leftrightarrow \Rpn \rho$.
Notice also that 
$$
h_R(R\rho\ta ,R\rho\ta+M\by) = \one \{ \rho \ge 1 \}\, h_1 (\rho\ta + M\by/R). 
$$
Combining all of these together we obtain
\begin{equation}\label{eq:exp1}
\begin{split}
\E \bigl( |\Nnkp \cap A| \bigr)  &= \frac{n^p}{p!}\, M^{d(p-1)} \Rpn^d (f(\Rpn))^p \int_1^\infty \rho^{d-1}  d\rho \int_{S^{d-1}} \hspace{-10pt}J(\ta)d\ta \int_{(\bbr^d)^{p-1}} \hspace{-5pt} d\by  \\
  & \times h_1(\rho\theta + M\by/R) g_1(0,\by)\mu_{(0,\by)}^{(k)}(A) \\
&  \times e^{-n \pM(R\rho\theta, R\rho\theta+M\by)} \frac{f(R\rho)}{f(R)}\prod_{i=1}^{p-1} \frac{f\big( R\| \rho \ta + My_i/R \| \big)}{f(R)}\\
%&\qquad \times \one \bigl\{ \min_{1\leq j \leq p-1} \| \rho\ta + M_n y_j /\Rpn \| \geq 1\bigr\}\, \one \bigl\{ \cechone (0,\by) \text{ is connected} \bigr\} \notag \\
%&\qquad \times  \muk_{(0,\by)}(A) \, e^{-n \pM (\Rpn \rho \ta, \Rpn \rho \ta +M_n \by)} \notag  \\
%&\qquad \times \frac{f(\Rpn \rho e_1)}{f(\Rpn e_1)}\, \prod_{j=1}^{p-1} \frac{f\bigl( \Rpn \| \rho \ta +M_n y_j/\Rpn \|e_1 \bigr)}{f(\Rpn e_1)}, \notag
\end{split}
\end{equation}
where $J(\ta)$ is the Jacobian given by \eqref{e:Jacobian}. 

Our next goal is to find the limit of the individual terms inside the integral.
First, notice that since $M/R\to 0$ we have that $h_1(\rho\theta+M\by/R)\to1$ for all $\rho \ge 1$, $\ta \in S^{d-1}$, and $\by = (y_1,\dots,y_{p-1}) \in (\bbr^d)^{p-1}$.
Next, appealing to the regular variation of $f$ in \eqref{e:RV}, we have
$$
\frac{f(R\rho)}{f(R)}\prod_{i=1}^{p-1} \frac{f\big( R\| \rho \ta + My_i/R \| \big)}{f(R)} \to \rho^{-\alpha p},
$$
for all $\rho\geq1$, $\ta\in S^{d-1}$, and $\by\in (\bbr^d)^{p-1}$.

Finally, we verify that the exponential term in \eqref{eq:exp1} converges to one.
To evaluate $Q_{2M}$ we apply the change of variable $z\leftrightarrow\Rpn \rho \ta + M v$ in \eqref{eq:Q}. This yields
\begin{align*}
&n\pM(\Rpn \rho \ta, \Rpn \rho \ta +M \by) \\
&\quad =nM^d f(\Rpn) \int_{\mathcal B_2(0,\by)} \frac{f\bigl( \Rpn \| \rho \ta +M v/\Rpn \| \bigr)}{f(\Rpn)}\, dv \\
&\quad \leq nM^d f(\Rpn) \sup_{v\in\mathcal B_2(0,\by)} \frac{f\bigl( \Rpn \| \rho \ta +M v/\Rpn \| \bigr)}{f(\Rpn)}\, \lambda_d\bigl( \mathcal B_2(0,\by) \bigr). 
\end{align*} 
Observe that for all $\rho \geq 1$, $\ta\in S^{d-1}$, and $v\in \mathcal B_2(0,\by)$ such that $\cechone (0,\by)$ is connected, we have, for large enough $n$,
$$
\| \rho \ta +M v/\Rpn \| \geq \frac{\rho}{2} \geq \frac{1}{2}. 
$$
Therefore, the Potter bound for regularly varying functions (e.g., Theorem 1.5.6 in \cite{bingham:goldie:teugels:1987} or Proposition 2.6 in \cite{resnick:2007}) gives, for every $0<\zeta <\alpha-d$,
\begin{align*}
&\sup_{v\in\mathcal B_2(0,\by)} \frac{f\bigl( \Rpn \| \rho \ta +M v/\Rpn \| \bigr)}{f(\Rpn )} \\
&\quad \leq C^* \sup_{v\in \mathcal B_2(0,\by)}\max \Bigl\{ \| \rho \ta +M v/\Rpn \|^{-\alpha +\zeta}, \, \| \rho \ta +M v/\Rpn \|^{-\alpha-\zeta}  \Bigr\} \\
&\quad \leq C^*. 
\end{align*}
Thus,  for all $\rho,\ta,\by$ we have
\[
n \pM(R\rho\theta, R\rho\theta+M\by) \le C^* nM^d f(R).
\]

Recalling \eqref{e:asym.equ}, together with the assumption $M/R\to 0$, ensures that $nM_n^df(\Rpn ) \to 0$, from which we can conclude that
\[
e^{-n \pM(R\rho\theta, R\rho\theta+M\by)} \to 1.
\]

Assuming that the dominated convergence theorem applies (as justified next), while using \eqref{e:asym.equ}, we can conclude that 
\[
\begin{split}
\E \bigl( |\Nnkp \cap A| \bigr) &\to \frac{1}{p!}\, \int_1^\infty \rho^{d-1-\alpha p}d\rho \int_{S^{d-1}} \hspace{-10pt}J(\ta)d\ta \int_{(\bbr^d)^{p-1}} \hspace{-5pt}  
g_1(0,\by)  \muk_{(0,\by)}(A) d\by\\
&= \frac{s_{d-1}}{p!(\alpha p -d)} \int_{(\R^d)^{p-1}} g_1(0,\by) \mu_{(0,\by)}^{(k)}(A) d\by\\
&= \E \bigl( \Nkp (A) \bigr), \ \ \ n\to\infty, 
\end{split}
\]
as required. 

It now remains to establish an integrable upper bound for an integrand in \eqref{eq:exp1}, in order to apply the dominated convergence theorem. First, the exponential term in \eqref{eq:exp1}
 is obviously bounded by one. As for the ratio of the densities, applying the Potter bound repeatedly we derive that, for every $0 < \zeta < \alpha -d$, there exists a $C>0$ (we have introduced a specific constant $C$, not a generic one, for later use) such that, for sufficiently large $n$, 
\begin{align} 
\frac{f(\Rpn \rho )}{f(\Rpn )} \one \{ \rho \geq 1 \} &\leq C \max \{ \rho^{-\alpha + \zeta},\, \rho^{-\alpha -\zeta} \}\, \one \{  \rho \geq 1\} \label{e:Potter1} \\
&= C \rho^{-\alpha +\zeta}\, \one \{  \rho \geq 1\} \notag,
\end{align}
and for each $i=1,\dots,p-1$,
\begin{align}  
& \frac{f\bigl( \Rpn \| \rho \ta +M y_i/\Rpn \| \bigr)}{f(\Rpn )}\, \one \bigl\{  \| \rho \ta +M y_i/\Rpn \| \geq 1 \bigr\} \label{e:Potter2}\\
&\leq C \max \Bigl\{  \| \rho \ta +M y_i/\Rpn \|^{-\alpha +\zeta}, \, \| \rho \ta +M y_i/\Rpn \|^{-\alpha -\zeta} \Bigr\} \notag \\
&\qquad \qquad \qquad \qquad \qquad \times \one \bigl\{  \| \rho \ta +M y_i/\Rpn \| \geq 1 \bigr\} \notag \\
&\leq C. \notag
\end{align}
Since
$$
\int_1^\infty \rho^{d-1-\alpha+\zeta} d\rho \int_{(\bbr^d)^{p-1}} g_1(0,\by)  \muk_{(0,\by)}(A)d\by <\infty, 
$$
we now obtain the required integrable bound. 
\end{proof}
\vspace{10pt}

\begin{proof}[Proof of Lemma \ref{l:exp.variance}]
Repeating the arguments in the proof of Lemma \ref{l:expectation}, we can write
\begin{align*}
\E \bigl(| \Nnkpone \cap A| \bigr)  &=  \frac{n^{p-1}}{(p-1)!}\, M^{d(p-2)} \Rpn^d (f(\Rpn))^{p-1} \int_1^\infty \rho^{d-1} d\rho \int_{S^{d-1}} \hspace{-10pt}J(\ta)d\ta \int_{(\bbr^d)^{p-2}} \hspace{-5pt} d\by  \\
&\quad \times h_1(\rho\theta + M\by /R) g_1(0,\by) 
  \muk_{(0,\by)}(A) \\
&\quad \times e^{-n \pM (\Rpn \rho \ta, \Rpn \rho \ta +M \by)} 
  \frac{f(\Rpn \rho )}{f(R)} \prod_{i=1}^{p-2} \frac{f\big( R \| \rho \ta + My_i /R \| \big)}{f(R)}. 
\end{align*}
From here, proceeding the same way as in the previous proof, we can conclude that the triple integral above converges to a positive constant. We also use \eqref{e:asym.equ} to get that, for some $C_1>0$,  
\begin{equation}  \label{e:asym.p-1}
\E \bigl(| \Nnkpone \cap A| \bigr) \sim C_1(n M^d f(R))^{-1}, \ \ \ n\to\infty. 
\end{equation}  

For the result on variance, we begin with writing 
\begin{align*}
\E \Bigl[ |\Nnkpone \cap A|^2 \Bigr] &= \sum_{\ell=0}^{p-1} \E \biggl[  \sum_{\substack{\Y \subset \Pn  \\  |\Y|  = |\Y'|= p-1}} \sum_{\substack{\Yp \subset \Pn  \\  | \Y \cap \Yp | =\ell}} h_R(\Y\cup\Yp)g_M(\Y,\Pn) g_M(\Yp, \Pn) \muk_\Y (A)\muk_\Yp (A) \biggl]\\
&=: \sum_{\ell=0}^{p-1} I_\ell. 
\end{align*}
For $\ell=p-1$, we know from \eqref{e:asym.p-1} that $I_{p-1} \sim C_1 (n M^d f(R))^{-1}$ as $n\to\infty$. 
For every $\ell \in \{1,\dots,p-2  \}$,  the condition $| \Y \cap \Yp | = \ell$ requires $g_M(\Y,\Pn)g_M(\Yp,\Pn) = 0$, as it is impossible for $\Y$ and $\Y'$ to be a connected component simultaneously.
Therefore, we have
\begin{align*}
\text{Var} \bigl( |\Nnkpone \cap A| \bigr) &= \sum_{\ell=0}^{p-1}I_\ell - \Bigl[ \E \bigl( |\Nnkpone \cap A| \bigr) \Bigr]^2 \\
&\sim C^* \bigl( nM^d f(\Rpn ) \bigr)^{-1} + I_0 - \Bigl[ \E \bigl(| \Nnkpone \cap A| \bigr) \Bigr]^2.
\end{align*}
To finish the proof we thus need to show that
$$
I_0 - \Bigl[ \E \bigl(| \Nnkpone \cap A| \bigr) \Bigr]^2 \to 0, \ \ \ n\to\infty. 
$$
Applying Palm theory yields  
\begin{align*}
I_0 &= \frac{n^{2(p-1)}}{\bigl( (p-1)! \bigr)^2}\, \E \Bigl[ h_R(\Y_{12}) g_M(\Y_1,\Y_{12}\cup \Pn)g_M(\Y_2,\Y_{12}\cup \Pn)\muk_{\Y_1}(A)\muk_{\Y_2}(A)  \Bigr],
\end{align*}
where $\Y_1$ and $\Y_2$ are disjoint sets of $(p-1)$ iid points respectively, and $\Y_{12} := \Y_1 \cup \Y_2$ is independent of $\Pn$. 
Conditioning on $\Y_{12}$, we have
\begin{align*}
I_0 &= \frac{n^{2(p-1)}}{\bigl( (p-1)! \bigr)^2}\, \E \Bigl[ h_R(\Y_{12}) g_M(\Y_1)g_M(\Y_2)\muk_{\Y_1} (A)\muk_{\Y_2} (A)\\
&\qquad \qquad \qquad \quad \times \one \bigl\{ \mathcal B_{M} (\Y_1) \cap \mathcal B_{M} (\Y_2) =\emptyset  \bigr\} e^{-n \pM (\Y_{12})} \Bigr]. 
\end{align*}
On the other hand, 
\begin{align*}
\Bigl[ &\E \bigl( |\Nnkpone \cap A| \bigr) \Bigr]^2 \\
&= \frac{n^{2(p-1)}}{\bigl( (p-1)! \bigr)^2}\, \E \Bigl[h_R(\Y_{12}) g_M(\Y_1) g_M(\Y_2)\muk_{\Y_1} (A)\muk_{\Y_2} (A)   e^{-n( \pM (\Y_1) + \pM (\Y_2))} \Bigr]. 
\end{align*}
Combining them together, we have
\begin{align*}
&I_0 - \Bigl[ \E \bigl(| \Nnkpone \cap A| \bigr) \Bigr]^2 \leq   \E (\Xi_n),  \label{e:0th.term} 
\end{align*}
where 
\begin{align*}
\Xi_n &= \frac{n^{2(p-1)}}{\bigl( (p-1)! \bigr)^2}\,h_R(\Y_{12}) g_M(\Y_1)g_M(\Y_2) \\
&\quad \times \muk_{\Y_1} (A)\muk_{\Y_2} (A)\, \Bigl( e^{-n \pM (\Y_{12})} - e^{-n (\pM(\Y_1) + n\pM (\Y_2))} \Bigr). 
\end{align*}
Furthermore $\E (\Xi_n)$ can be split into two parts, 
$$
\E \Bigl[ \, \Xi_n \Bigl( \one \bigl\{ \BM (\Y_1) \cap \BM(\Y_2) =\emptyset \bigr\}  +\one \bigl\{ \BM (\Y_1) \cap \BM(\Y_2) \neq \emptyset \bigr\} \Bigr) \Bigr].
$$
Note that whenever $\BM(\Y_1)\cap \BM(\Y_2) = \emptyset$, we have $\pM(\Y_{12}) = \pM(\Y_1) + \pM(\Y_2)$, in which case  $\Xi_n=0$. So it suffices to consider the other part only. Bounding an exponential term by one, 
\begin{align*}
\E[\Xi_n] &\leq \frac{n^{2(p-1)}}{\bigl( (p-1)! \bigr)^2}\, \E \Bigl[h_R(\Y_{12})g_M(\Y_1)g_M(\Y_2)
 \muk_{\Y_1} (A)\muk_{\Y_2} (A)\, \one \bigl\{ \BM (\Y_1) \cap \BM(\Y_2) \neq \emptyset \bigr\}  \Bigr].
\end{align*}
Notice that
\begin{align*}
g_M(\Y_1)g_M(\Y_2)\one \bigl\{ \BM (\Y_1) \cap \BM(\Y_2) \neq \emptyset \bigr\} 
\leq g_{2M}(\Y_{12}).
\end{align*}
This, together with the fact that $\mu_{\Y_i}^{(k)}(A) \le \binom{p-1}{k+1}$, yields
\begin{equation}  \label{e:final.bound.EDelta}
\E[\Xi_n] \leq \frac{n^{2(p-1)}}{\bigl( (p-1)! \bigr)^2}\, \begin{pmatrix} p-1 \\ k+1 \end{pmatrix}^2 \E \bigl[ h_R(\Y_{12})  g_{2M}(\Y_{12})   \bigr].
\end{equation}
Calculating the expectation portion as in the proof of Lemma \ref{l:expectation} and using \eqref{e:asym.equ}, we find that the right hand side in \eqref{e:final.bound.EDelta} equals
\begin{align*}
 O \bigl( n^{2(p-1)} M^{d(2p-3)} \Rpn^d(f(\Rpn))^{2(p-1)} \bigr) =  O\Bigl( \bigl( nM^df(\Rpn ) \bigr)^{p-2} \Bigr) \to0, \ \ \ n\to\infty. 
\end{align*}
Now, the entire proof has been completed. 
\end{proof}
\vspace{10pt}

\begin{proof}[Proof of Theorem \ref{t:main.fell}]
We divide the proof into three parts.\\
\noindent
{\bf Part I} - Prove the ``random" part of the limit, i.e.
\begin{equation}  \label{e:main.part1}
\bigcup_{m=p}^\infty \Nnkm {\Rightarrow} \Nkp \ \ \text{in } \mathcal F(\Delta). 
\end{equation}

\noindent
{\bf Part II} - Prove the ``nonrandom" part of the limit, i.e.
$$
\bigcup_{m=k+2}^{p-1} \Nnkm \Rightarrow B_{k,p-1} \  \ \text{in } \mathcal F(\Delta). 
$$

\noindent
{\bf Part III} - Combine I and II to conclude the statement in the theorem, 
$$
\NnRk \Rightarrow \Nkp \cup B_{k,p-1} \ \ \text{in } \mathcal F(\Delta). 
$$

%%%%%%%%%%

\noindent {\bf Part I}: We wish to prove \eqref{e:main.part1}. 
By virtue of Theorem 6.5 in \cite{molchanov:2005}, it is enough to verify that for every compact subset $A \subset \Delta$ with $A\cap B_{k,p}\neq \emptyset$,  we have
$$
\P \Bigl( \bigcup_{m=p}^\infty \Nnkm \cap A \neq \emptyset \Bigr) \to \P \bigl( \Nkp \cap A\neq \emptyset \bigr). 
$$
We can proceed as follows: 
\begin{equation}\label{eq:t123}
\begin{split}
&\Bigl| \, \P \Bigl( \bigcup_{m=p}^\infty \Nnkm \cap A \neq \emptyset \Bigr) -  \P \bigl( \Nkp \cap A\neq \emptyset \bigr) \Bigr| \\
& \leq  \Bigl| \, \P\bigl( \Nnkp \cap A =\emptyset \bigr) - \P \bigl( \Nkp \cap A=\emptyset \bigr) \Bigr| + \P \Bigl( \bigcup_{m=p+1}^\infty \big\{ \Nnkm \cap A \neq \emptyset  \big\}\Bigr) \\
&\leq  \Bigl| \, \P\bigl( \Ntnkp \cap A =\emptyset \bigr) - \P \bigl( \Nkp \cap A =\emptyset \bigr) \Bigr| + \P \bigl( ( \Ntnkp \setminus \Nnkp ) \cap A \neq \emptyset \bigr)\\
&\qquad\qquad\qquad\qquad\qquad\qquad\qquad \qquad\qquad + \P \Bigl( \bigcup_{m=p+1}^\infty \big\{ \Nnkm \cap A \neq \emptyset  \big\} \Bigr) \\
&\leq \Bigl| \, \P\bigl( \Ntnkp \cap A =\emptyset \bigr) - \P \bigl( \Nkp \cap A =\emptyset \bigr) \Bigr|+ \E \Bigl[| \Ntnkp \cap A| - |\Nnkp \cap A| \Bigr] \\
&\qquad\qquad\qquad\qquad\qquad\qquad\qquad \qquad\qquad +  \sum_{m=p+1}^\infty \E\bigl( | \Nnkm \cap A|\bigr),\\
& =: T_1 + T_2 + T_3.
\end{split}
\end{equation}
where the last step follows from Markov's inequality. 
To complete the proof, we thus need to show that $T_i \to 0$ for $i=1,2,3$.
First, $T_2\to 0$ follows as a direct consequence of Lemma \ref{l:expectation}.

Next, we show that $T_1\to 0$. To this end, we introduce an iid random sample version of $\Ntnkp$. More specifically,  let 
$$
\I_n := \begin{cases} 
\bigl\{ (i_1,\dots,i_p) \in \bbn_+^p: 1 \leq i_1 < \cdots < i_p \leq n \bigr\} & \text{if } \ n \geq p, \\
\emptyset & \text{if } \ n <p.
\end{cases}
$$
and 
$$
\eta_{\bi, n} := h_R(\X_\bi) g_M(\X_\bi)  \muk_{\X_\bi} (M A), 
$$
where $\X_\bi = (X_{i_1}, \dots, X_{i_p})$ for $\bi=(i_1,\dots, i_p) \in \I_n$. Notice that the $X_i$'s here are the same as those in \eqref{eq:pois} to  generate $\cP_n$.
In other words, $\eta_{\bi,n}$ is the total number of the $k$th persistence pairs lying in $MA$ and generated by the subset $\X_{\bi}$, with the restriction that $\X_{\bi}$ is connected and each point in $\X_{\bi}$ lies outside $B(0; R)$.
We now claim that
$$
\P \Bigl( \sum_{\bi \in \I_n} \eta_{\bi, n} = 0 \Bigr)-\P \bigl( \Ntnkp \cap A = \emptyset\bigr)  \to 0, \ \ \ n\to\infty.
$$
For any integer-valued random variables $Y_1$ and $Y_2$ defined on the same probability space we have,
$$
\bigl|\, \P (Y_1=0) - \P(Y_2=0) \bigr| \leq \E \bigl( | Y_1-Y_2 | \bigr). 
$$
Therefore,
\begin{align}
\Bigl|\,  \P \Bigl( \sum_{\bi \in \I_n} \eta_{\bi, n} = 0 \Bigr) - \P \bigl( \Ntnkp \cap A = \emptyset \bigr) \Bigr| &\leq \E \Bigl[ \bigl|  \sum_{\bi \in \I_n} \eta_{\bi, n} - |\Ntnkp \cap A| \bigr| \Bigr] \label{e:Poisson.iid} \\
&= \sum_{m=0}^\infty \P \bigl( |\Pn| = m \bigr) \E \Bigl[ \, \Bigl|  \sum_{\bi \in \I_n} \eta_{\bi, n} - \sum_{\bi \in \I_m} \eta_{\bi, n} \Bigr|  \, \Bigr]  \notag \\
&= \sum_{m=0}^\infty \P \bigl( |\Pn| = m \bigr) \Bigl| \begin{pmatrix} n\\p \end{pmatrix} - \begin{pmatrix} m\\p \end{pmatrix} \Bigr| \, \E(\eta_{\bi,n}). \notag 
\end{align}
Returning to \eqref{e:palm.theory.first.lemma}, we find that the expectation portion of the right hand side in \eqref{e:palm.theory.first.lemma} is asymptotically equal to $\E(\eta_{\bi, n})$. Additionally, Lemma \ref{l:expectation} ensures that $\E \big(| \Nnkp \cap A| \big)$ tends to a positive constant as $n\to\infty$. Hence the rightmost term at \eqref{e:Poisson.iid} can be bounded by 
\begin{equation}  \label{e:binom.bound}
C^*  \sum_{m=0}^\infty \P \bigl( |\Pn| = m \bigr)  n^{-p} \Bigl| \begin{pmatrix} n\\p \end{pmatrix} - \begin{pmatrix} m\\p \end{pmatrix} \Bigr|,
\end{equation}
which itself is further bounded by 
\begin{align*}
C^* \bigg\{ n^{-p} \Big| \binom{n}{p} -\frac{n^p}{p!} \Big| + \frac{1}{p!}\, \E \bigg[ \Big| \Big( \frac{|\Pn|}{n}  \Big)^p -1 \, \Big|  \bigg]  +n^{-p} \sum_{m=0}^\infty \P(|\Pn| = m) \Big| \binom{m}{p} -\frac{m^p}{p!} \Big| \bigg\}. 
\end{align*}
It is now straightforward to show that each of the three terms converges to $0$ as $n\to\infty$. 

To prove $T_1\to 0$, it now suffices to show that 
\begin{equation}  \label{e:dePoisson}
\P \Bigl( \sum_{\bi \in \I_n} \eta_{\bi, n} = 0 \Bigr) - \P \bigl( \Nkp \cap A = \emptyset \bigr) \to 0, \ \ \ n\to\infty. 
\end{equation}
To this end, our argument relies on the so-called \textit{total variation distance}, which is defined for two random variables $Y_1, Y_2$ as 
$$
d_{\text{TV}} (Y_1,Y_2) := \sup_{A\subset \bbr} \bigl|\, \P(Y_1\in A) - \P(Y_2\in A) \bigr|. 
$$
Denoting $Z\sim \text{Poisson}\Bigl(  \E \bigl( \sum_{\bi \in \I_n} \eta_{\bi,n}\bigr) \Bigr)$, and using the triangle inequality, we have 
\begin{align*}
\Bigl|\,  &\P \Bigl( \sum_{\bi \in \I_n} \eta_{\bi, n} = 0 \Bigr) - \P \bigl( \Nkp \cap A = \emptyset \bigr) \Bigr| \\
&\leq d_{\text{TV}} \biggl( \sum_{\bi \in \I_n}\eta_{\bi, n}, \, Z \biggr)  + \biggl|\, \P (  Z =0 ) - \P\bigl( \Nkp \cap A =\emptyset \bigr)  \biggr| 
\end{align*}
Since $Z$ and $\Nkp(A)$ are both Poisson, an elementary calculation shows that 
\begin{align*}
\biggl|\, \P (  Z =0 ) - \P\bigl( \Nkp \cap A=\emptyset \bigr) \biggr| &\leq \Bigl| \,  \E \Bigl( \sum_{\bi \in \I_n} \eta_{\bi, n} \Bigr) - \E \bigl(| \Nkp \cap A| \bigr)  \Bigr| \\
&\leq \Bigl| \,  \E \bigl(| \Ntnkp \cap A| \bigr) - \E \bigl(| \Nkp \cap A| \bigr) \Bigr| + o(1) \to 0,
\end{align*}
where we have used \eqref{e:Poisson.iid} and Lemma \ref{l:expectation}. 

In order to bound $d_{\text{TV}} ( \sum_{\bi \in \I_n}\eta_{\bi, n}, \, Z )$ we will use  \textit{Stein's Poisson approximation theorem} (e.g., Theorem 2.1 in \cite{penrose:2003}). As preparation, however, we need to define a certain graph on $\I_n$ as follows. For $\bi, \bj \in \I_n$, write $\bi \sim \bj$ if and only if they have at least one common element, i.e., $| \bi \cap \bj | > 0$. Then, $(\I_n, \sim)$ constitutes a \textit{dependency graph}, that is, for every $I_1, I_2 \subset \I_n$ with no edges connecting $I_1$ and $I_2$, we have that $(\eta_{\bi, n}, \, \bi \in I_1)$ and $(\eta_{\bi, n}, \, \bi \in I_2)$ are independent. 
Under this setup, Stein's Poisson approximation theorem yields 
\[d_{\text{TV}} \biggl( \sum_{\bi \in \I_n}\eta_{\bi, n}, Z \biggr) \leq 3 \biggl[  \, \sum_{\bi \in \I_n} \sum_{\bj \in N_\bi} \E (\eta_{\bi,n}) \E(\eta_{\bj, n}) + \sum_{\bi \in \I_n} \sum_{\bj \in N_\bi\setminus \{  \bi\}} \E(\eta_{\bi,n}\eta_{\bj,n})  \, \biggr], 
\]
where $N_{\bi} = \{ \bj \in \I_n: \bi \sim \bj \} \cup \{ \bi \}$. \\
From the argument before \eqref{e:binom.bound}, we know that for sufficiently large $n$,
$$
\E (\eta_{\bi,n}) \leq 2 \matnp^{-1} \E \bigl(| \Nkp \cap A| \bigr).
$$
Therefore,
\begin{align*}
&\sum_{\bi \in \I_n} \sum_{\bj \in N_\bi} \E (\eta_{\bi,n}) \E(\eta_{\bj, n}) \\
&\leq \matnp \biggl( \matnp - \begin{pmatrix} n-p \\ p \end{pmatrix} \biggr) 4 \matnp^{-2} \Bigl( \E \bigl(| \Nkp\cap A| \bigr) \Bigr)^2  \to 0 \ \ \text{as } n\to\infty. 
\end{align*}
For $\bi, \bj\in \I_n$ with $\ell := | \bi \cap \bj | \in \{ 1,\dots,p-1 \}$, by the same change of variables as in the proof of Lemma \ref{l:expectation}, we have 
\begin{align*}
\E \bigl(\eta_{\bi, n}\eta_{\bj,n} \bigr)  &=  M^{d(2p-\ell-1)} \Rpn^d (f(\Rpn))^{2p-\ell} \int_1^\infty  \rho^{d-1} d\rho \int_{S^{d-1}} \hspace{-10pt}J(\ta)d\ta \int_{(\bbr^d)^{2p-\ell-1}} \hspace{-5pt} d\by  \\
& \times h_1(\rho\theta+M\by/\Rpn)g_1(0,\by)   \muk_{(0,\by)}(A) \\
&\times e^{-n \pM (\Rpn \rho \ta, \Rpn \rho \ta +M \by)} \frac{f(R\rho)}{f(R)} \prod_{i=1}^{2p-\ell-1} \frac{f\big(  R \| \rho \ta + M y_i /R \|\big)}{f(R)}\\
&= O \bigl( M^{d(2p-\ell-1)} \Rpn^d (f(\Rpn ))^{2p-\ell} \bigr). 
\end{align*}
It follows from \eqref{e:asym.equ} that 
\begin{align*}
\sum_{\bi \in \I_n}\sum_{\bj \in N_\bi \setminus \{ \bi \}} \E (\eta_{\bi, n}\eta_{\bi,n}) 
&= \sum_{\ell=1}^{p-1} \matnp \begin{pmatrix} p\\ \ell \end{pmatrix} \begin{pmatrix} n-p\\ p-\ell \end{pmatrix} \E(\eta_{\bi,n}\eta_{\bj,n})\, \one \bigl\{ |\bi \cap \bj| =\ell \bigr\} \\
&\leq C^*\sum_{\ell=1}^{p-1} n^{2p-\ell} M_n^{d(2p-\ell-1)} \Rpn^d (f(\Rpn))^{2p-\ell} \\
&\leq C^* nM_n^d f(\Rpn) \to 0, \ \ \ n\to\infty,
\end{align*}
and hence, \eqref{e:dePoisson} is established. 

Finally we turn our attention to showing that $T_3\to 0$ in \eqref{eq:t123}. For every $m \geq p+1$, repeating the argument of Lemma \ref{l:expectation}, 
\begin{align*}
\E \bigl( |\Nnkm \cap A| \bigr)  &=  \frac{n^m}{m!}\, M^{d(m-1)} \Rpn^d (f(\Rpn))^{m} \int_1^\infty \rho^{d-1} d\rho \int_{S^{d-1}} \hspace{-10pt}J(\ta)d\ta \int_{(\bbr^d)^{m-1}} \hspace{-5pt} d\by  \\
&\times h_1( \rho\ta + M \by /\Rpn )g_1(0,\by)   \muk_{(0,\by)}(A) \\
&\times e^{-n \pM (\Rpn \rho \ta, \Rpn \rho \ta +M \by)} \frac{f(R\rho)}{f(R)} \prod_{i=1}^{m-1} \frac{f\big( R \| \rho \ta + My_i /R \|  \big)}{f(R)}. 
\end{align*}
Using \eqref{e:Potter1} and \eqref{e:Potter2}, together with the fact that for large enough $n$
$$
n^m M^{d(m-1)} \Rpn^d (f(\Rpn ))^{m} \leq 2\bigl( nM^df(\Rpn ) \bigr)^{m-p},
$$
we have
\begin{align*}
\E \bigl(| \Nnkm \cap A| \bigr) &\leq \frac{2C^m}{m!}\, \bigl( nM^df(\Rpn ) \bigr)^{m-p}\\
& \times \int_1^\infty \rho^{d-1-\alpha +\zeta} d\rho \int_{S^{d-1}} \hspace{-10pt}J(\ta)d\ta \int_{(\bbr^d)^{m-1}} \hspace{-10pt} d\by   g_1(0,\by) \muk_{(0,\by)}(A) \\
&\leq \frac{2s_{d-1}}{\alpha - d-\zeta}\, \frac{C^m}{m!}\, \begin{pmatrix} m \\ k+1 \end{pmatrix} \bigl( nM^df(\Rpn ) \bigr)^{m-p}  \int_{(\bbr^d)^{m-1}} g_1(0,\by) d\by, 
\end{align*}
where at the second step we have applied $\mu_{(0,\by)}^{(k)}(A) \le \binom{m}{k+1}$. 
Next, the well-known fact that there exist $m^{m-2}$ spanning trees on a set of $m$ vertices, yields 
$$
\int_{(\bbr^d)^{m-1}} g_1 (0,\by)  d\by \leq m^{m-2} \omega_d^{m-1},
$$
where $\omega_d$ represents the volume of a unit ball in $\bbr^d$. 
To show that $T_3\to 0$, it  therefore remains to verify that
$$
\sum_{m=p+1}^\infty \frac{C^m}{m!}\, \begin{pmatrix}
m \\k+1
\end{pmatrix} \bigl( nM^df(\Rpn ) \bigr)^{m-p} m^{m-2} \omega_d^{m-1} \to 0, \ \ \ n\to\infty.
$$
From Stirling's formula (i.e., $m!\geq (m/e)^m$ for large $m$ enough), we can bound the left hand side above by a constant multiple of
$$
\sum_{m=p+1}^\infty m^{k-1} \bigl( e C \omega_d nM^df(\Rpn ) \bigr)^{m-p}, 
$$
which clearly vanishes as $n\to\infty$. We thus proved that $T_i\to 0$ for $i=1,2,3$ in \eqref{eq:t123}, so we can conclude Part I. 
\vspace{10pt}

%%%%%%%%%%

\noindent {\bf{Part II}:} \\
Recall that our goal here is to prove  the nonrandom part of the limit, i.e.
\begin{equation}\label{e:goal.part2}
\Nhnkp := \bigcup_{m=k+2}^{p-1} \Nnkm \Rightarrow B_{k,p-1} \ \ \text{in } \mathcal F(\Delta). 
\end{equation}
First, for a measurable set $A\subset \Delta$ and $\epsilon >0$, denote by $(A)^{\epsilon- }$ 
an open $\epsilon$-envelop in terms of the Euclidean metric (see \eqref{e:open.envelop}). By the definition of convergence in probability under the Fell topology (see Definition 6.19 in \cite{molchanov:2005}), we need to show that 
$$
\P \biggl( \Bigl[ \Bigl(\Nhnkp\setminus (B_{k,p-1})^{\epsilon-} \Bigr) \cup \Bigl( B_{k,p-1} \setminus \Bigl( \Nhnkp \Bigr)^{\epsilon-}   \Bigr) \Bigr] \cap K \neq \emptyset  \biggr) \to 0
$$
for every $\epsilon>0$ and a compact set $K$ in $\Delta$. Since by construction, $\Nhnkp \subset B_{k,p-1}$, we have
$$
\Nhnkp \setminus (B_{k,p-1})^{\epsilon-} = \emptyset \ \ \text{a.s.}
$$
It thus remains to prove that
$$
\P \biggl( \Bigl[ B_{k,p-1}  \setminus \Bigl(  \Nhnkp \Bigr)^{\epsilon-}   \Bigr] \cap K \neq \emptyset  \biggr) \to 0, \ \ \ n\to\infty. 
$$
Note that $\bigl(\Nhnkp\bigr)^{\epsilon-}$ 
is the union of open balls of radius $\epsilon$ centered about the points in $\Nhnkp$. 
Since $K$ is a closed and bounded set, we can take, without loss of generality, $K=\bigl( [a,b] \times \bbr_+ \bigr) \cap B_{k,p-1} $ for some $0\leq a<b<\infty$. Let $\Q_{p-1}$ be a collection of cubes in $\bbr^2$ with side length $\epsilon/\sqrt{2}$ such that each cube intersects with $K$ and the union of these cubes covers $K$. Then
\begin{align*}
&\P \biggl( \Bigl[ (B_{k,p-1} ) \setminus \Bigl( \Nhnkp\Bigr)^{\epsilon-}   \Bigr] \cap K \neq \emptyset  \biggr) \\
&\leq \P \Bigl( \bigcup_{Q \in\Q_{p-1}} \bigcap_{m=k+2}^{p-1} \{ \Nnkm \cap Q = \emptyset \}  \Bigr) \\
&\leq \sum_{Q\in\Q_{p-1}} \P \bigl( \Nnkpone \cap Q =\emptyset \bigr) \\
&\leq \sum_{Q\in\Q_{p-1}} \P \Bigl( \, \bigl| \, |\Nnkpone \cap Q| - \E\bigl(| \Nnkpone \cap Q| \bigr)  \bigr| \geq \E \bigl(| \Nnkpone \cap Q| \bigr) \Bigr) \\
&\leq \sum_{Q\in\Q_{p-1}} \frac{\text{Var} \bigl(| \Nnkpone \cap Q| \bigr)}{\Bigl[ \E \bigl(  |\Nnkpone \cap Q| \bigr) \Bigr]^2}.
\end{align*}
Hence, from Lemma \ref{l:exp.variance} we can bound the rightmost term by a constant multiple of $nM^df(\Rpn )$. 
Since $nM^df(\Rpn ) \to 0$, the desired result follows. 
\vspace{10pt}

%%%%%%%%%%

\noindent {\bf Part III:}\\
Here we wish to combine I and II to conclude the statement in the theorem,
$$
\NnRk \Rightarrow \Nkp \cup B_{k,p-1} \ \ \text{in } \mathcal F(\Delta). 
$$
Since $\mathcal F(\Delta)$ is metrizable in the Fell topology (see \cite{attouch:1984}), \eqref{e:goal.part2} implies that there exists a metric on $\mathcal F(\Delta)$, denoted $\rho$, such that 
\begin{equation}  \label{e:rephrase.part2}
\rho \Bigl( \Nhnkp, \, B_{k,p-1}  \Bigr) \stackrel{p}{\to} 0. 
\end{equation}
Now, combining the convergences \eqref{e:main.part1} and \eqref{e:rephrase.part2} gives (see Proposition 3.1 in \cite{resnick:2007}), 
$$
\Bigl( \bigcup_{m=p}^\infty \Nnkm, \, \Nhnkp \Bigr) \Rightarrow \bigl( \Nkp, B_{k,p-1} \bigr) \ \ \text{in } \mathcal F(\Delta) \times \mathcal F(\Delta), 
$$
where $\mathcal F(\Delta) \times \mathcal F(\Delta)$ is equipped with the product topology. 
Finally, using the fact that 
$$
(F_1,F_2) \in \mathcal F(\Delta) \times \mathcal F(\Delta) \mapsto F_1 \cup F_2 \in \mathcal F(\Delta)
$$
is continuous (see page 7 in \cite{matheron:1975}), we can conclude from the continuous mapping theorem that
$$
\bigcup_{m=k+2}^{\infty} \Nnkm \Rightarrow \Nkp \cup B_{k,p-1} \ \ \text{in } \mathcal F(\Delta).
$$
\end{proof}

Before finishing this section we provide a proof for Corollary \ref{cor:vague}.
\begin{proof}[Proof of Corollary \ref{cor:vague}]
To show convergence of $\sum_{m=p}^\infty \Nnkm$ in $\text{MP}(\Delta)$, we will use Kallenberg's theorem (see Proposition 3.22 in \cite{resnick:1987}), for which we need to show that for any measurable set $A \subset \Delta$, 
\[
\begin{split}
\E \Bigl[\, \sum_{m=p}^\infty \Nnkm (A) \Bigr] &\to \E\bigl( \Nkp (A) \bigr), \\
\P \Bigl(\, \sum_{m=p}^\infty \Nnkm (A)=0 \Bigr) &\to \P \Bigl( \Nkp(A)=0 \Bigr). 
\end{split}
\]
The first limit is a direct result of Lemma \ref{l:expectation} and $T_3 \to 0$ in \eqref{eq:t123}. 
For the second limit, we have
\begin{align*}
\P \Bigl(\, \sum_{m=p}^\infty \Nnkm (A)=0 \Bigr) &= \P\bigl( \Nnkp(A)=0 \bigr) + o(1) \\
&= \P\bigl( \Ntnkp(A)=0 \bigr) + o(1) \\
&\to \P \bigl( \Nkp(A)=0 \bigr), \ \ \ n\to\infty,
\end{align*}
where we have used $T_i\to 0$, $i=1,2,3$ in \eqref{eq:t123}.
\end{proof}

\subsection{Exponentially decaying tails}  \label{sec:proof.exp.tail}

The proof for the exponentially decaying tail case goes mostly parallel to that in the previous subsection. 
In particular, regardless of heaviness of the tail of $f$, the weak limits in Theorems \ref{t:main.fell} and \ref{t:main.fell.exponential} are characterized by a Poisson random measure, the only difference lying in the limiting mean measures. Therefore, the current subsection only presents the results on the moment asymptotics corresponding to Lemmas \ref{l:expectation} and \ref{l:exp.variance}. All the arguments that follow are essentially the same as the heavy tail case,  so we omit them. 

\begin{lemma} \label{l:expectation.exponential}
Let $A\subset \Delta$ be a measurable set, such that $A \cap B_{k,p} \neq \emptyset$. Under the conditions of Theorem \ref{t:main.fell.exponential},
\[
\lim_{n\to\infty}\E \bigl( |\Nnkp \cap A| \bigr) = \lim_{n\to\infty}\E \bigl(| \Ntnkp \cap A| \bigr) = \E \bigl( |\Nkp \cap A| \bigr) \in (0,\infty).
\]
If $A\cap B_{k,p-1} \ne \emptyset$, then
\begin{align*}
\E\bigl( | \Nnkpone \cap A| \bigr) &\sim C_3\bigl( nM^df(\Rpn) \bigr)^{-1}, \ \ n\to\infty, \ \ \text{and } \\
\text{Var} \bigl( | \Nnkpone \cap A| \bigr) &\le C_4 \big( nM^d f(R) \big)^{-1}
\end{align*}
for some $C_3, C_4>0$, which are independent of $n$. 
\end{lemma}

\begin{proof}
Among these claims in the above lemma, we shall prove the first limit only, i.e. 
\begin{equation}  \label{e:first.limit.exp}
\lim_{n\to\infty}\E \bigl( |\Nnkp \cap A| \bigr) = \E \bigl( |\Nkp\cap A| \bigr).
\end{equation}
The rest of the proofs will be similar and hence omitted. 

Using Palm theory we have
\begin{align*}
\E \bigl( |\Nnkp \cap A| \bigr) &= \frac{n^p}{p!}\, \E \Bigl[ h_R(\X_p) g_M(\X_p,\X_p \cup \Pn) \muk_{\X_p} (MA)\Bigr],
\end{align*}
where $\X_p=(X_1,\dots,X_p)$ denotes iid points, independent of $\Pn$. 
Conditioning on $\X_p$ and changing variables $x_1 \leftrightarrow x$, $x_i  \leftrightarrow x+M y_{i-1}$, $i=2,\dots,p$, we obtain 
\[
\begin{split}
\E \bigl(| \Nnkp \cap A| \bigr)  = \frac{n^p}{p!}\, M^{d(p-1)} \int_{\bbr^d}dx \int_{(\bbr^d)^{p-1}} d\by & h_R(x,x+M\by) g_1(0,\by) \mu_{(0,\by)}^{(k)}(A)\\
\times &e^{-n \pM(x,x+M\by)} f(x) \prod_{i=1}^{p-1} f(x + My_i). 
\end{split}
\]
Changing into polar coordinate change $x \leftrightarrow (r,\ta)$, along with an additional change of variable $\rho = a(\Rpn)^{-1} (r-\Rpn)$ we have 
\begin{equation}\label{e:long.expression}
\begin{split}
\E \bigl(| \Nnkp \cap A| \bigr) &= \frac{n^p}{p!}\, M^{d(p-1)} a(\Rpn) \Rpn^{d-1} (f(\Rpn))^p\int_0^\infty d\rho \int_{S^{d-1}} J(\ta) d\ta \int_{(\bbr^d)^{p-1}} d\by \\
&\times \Bigl( 1+\frac{a(\Rpn)}{\Rpn}\rho \Bigr)^{d-1} 
h_R\big((R+a(R)\rho)\theta,(R+a(R)\rho)\theta+M\by\big) \\
&\times g_1(0,\by)\muk_{(0,\by)}(A) e^{-n \Q_{2M} ( ( \Rpn +a(\Rpn)\rho )\ta,  (\Rpn + a(\Rpn)\rho)\ta + M \by)} \\
&\times \frac{f\big( R + a(R)\rho \big)}{f(R)} \prod_{i=1}^{p-1} \frac{f\big( \| (\Rpn + a(\Rpn)\rho)\ta + M y_i \| \big)}{f(R)}. 
%\frac{f\bigl( (\Rpn + a(\Rpn)\rho)  \bigr)f\bigl( \| (\Rpn + a(\Rpn)\rho)\ta + M \by \| \bigr)}{f(\Rpn e_1)^p}
\end{split}
\end{equation}

Using the Taylor expansion, we have
\begin{equation}  \label{e:Taylor1}
\| (\Rpn + a(\Rpn)\rho)\ta + M y_i \| = \Rpn + a(\Rpn) \rho + M \bigl( \langle \ta, y_i \rangle + \gamma_n (\rho, \ta, y_i) \bigr),
\end{equation}
where $\gamma_n (\rho, \ta, y_i) \to 0$ uniformly for $\rho >0$, $\ta \in S^{d-1}$, and $y_i$ in a bounded set in $\bbr^d$. Denoting
$$
\xi_n (\rho, \ta, y_i)  = \frac{\langle \ta, y_i \rangle + \gamma_n (\rho, \ta, y_i)}{a(\Rpn)/M}, 
$$
the right hand side in \eqref{e:Taylor1} is equal to 
$\Rpn + a(\Rpn) (\rho + \xi_n(\rho, \ta, y_i))$. 
Due to the uniform convergence of $\gamma_n (\rho, \ta, y_i)$, it is easy to show that for every $M>0$, 
\begin{equation}  \label{e:uniform.bound}
\sup_{\substack{\rho >0, n \geq 1,  \ta \in S^{d-1} \\ y_i \in [-M,M]^d}} \bigl| \xi_n (\rho, \ta, y_i) \bigr| < \infty, 
\end{equation}
and further, 
\begin{equation}  \label{e:conv.xi}
\xi_n (\rho, \ta, y_i) \to c^{-1} \langle \ta, y_i \rangle \ \ \text{as } n \to\infty. 
\end{equation}

In the following, we shall compute the limits for each term in \eqref{e:long.expression}  under the integral sign, and then establish an appropriate integrable bound for the application of the dominated convergence theorem. From \eqref{e:cesaro} we have that $\bigl( 1 + a(R)\rho/R \bigr) \to 1$ 
for all $\rho > 0$, and for sufficiently large $n$, this is bounded by $2 (\rho \vee 1)^{d-1}$. 
Subsequently, from \eqref{e:Taylor1} and \eqref{e:conv.xi}, we have that 
$$
h_R\big((R+a(R)\rho)\theta,(R+a(R)\rho)\theta+M\by\big)  \to \one \bigl\{ \rho + c^{-1} \langle \ta, y_i \rangle \geq 0, \ i=1,\dots,p-1 \bigr\}. 
$$
As for the ratio terms for the density $f$, using \eqref{eq:f_vm} we write
\begin{equation} \label{e:first.ratio}
\frac{f\bigl( \Rpn + a(\Rpn)\rho \bigr)}{f(\Rpn )} = \frac{L\bigl( \Rpn  + a(\Rpn)\rho \bigr)}{L(\Rpn)}\, e^{-[ \psi(\Rpn + a(\Rpn)\rho ) - \psi(\Rpn) ]}
\end{equation}
so that $L\bigl( \Rpn  + a(\Rpn)\rho \bigr) / L(\Rpn) \to 1$ 
for all $\rho>0$, and 
$$
e^{-[ \psi(\Rpn + a(\Rpn)\rho ) - \psi(\Rpn) ]} = \exp \Bigl\{  -\int_0^\rho \frac{a(\Rpn)}{a(\Rpn + a(\Rpn)r)}\, dr  \Bigr\} \to e^{-\rho} , \ \ \ n\to\infty.
$$
For the last convergence, we  applied an elementary result in p.~142 of \cite{embrechts:kluppelberg:mikosch:1997}, which asserts that 
$$
\frac{a(x)}{a\big( x + a(x)r \big)} \to 1 \ \ \text{as } x\to\infty,
$$
uniformly for $r$ in any bounded interval. 
In order to give a proper upper bound for \eqref{e:first.ratio}, 
we let 
$$
q_m(n) =  \frac{\psi^{-1} \bigl( \psi (\Rpn) + m \bigr) - \Rpn}{ a(\Rpn)}, \ \ m \ge 1,
$$
equivalently, $\psi \bigl( \Rpn  + a(\Rpn)q_m(n)  \bigr) = \psi (\Rpn) + m$. 
Accordingly to Lemma 5.2 in \cite{balkema:embrechts:2004}, for every $0 < \epsilon < (d+\gamma p)^{-1}$ ($\gamma$ is a parameter at \eqref{e:L.gamma}), there exists $N \geq 1$ such that $q_m (n) \leq e^{m \epsilon} / \epsilon$ for all $n \geq N$ and $m \geq 1$. 
Since $\psi$ is increasing, we have
\begin{align*}
e^{-[ \psi(\Rpn + a(\Rpn)\rho ) - \psi(\Rpn) ]}  \one\{ \rho >0 \} &= \sum_{m=0}^\infty \one \bigl\{ q_m(n) < \rho \leq q_{m+1}(n) \bigr\} e^{-[ \psi(\Rpn + a(\Rpn)\rho ) - \psi(\Rpn) ]} \\
&\leq \sum_{m=0}^\infty \one \bigl\{ 0 < \rho \leq e^{(m+1)\epsilon} / \epsilon \bigr\}\, e^{-m}
\end{align*}
for all $n \geq N$. 
For the derivation of the bound for $L$, we use \eqref{e:L.gamma}, i.e., 
$$
\frac{L\bigl( \Rpn  + a(\Rpn)\rho \bigr)}{L(\Rpn)}\, \one \{ \rho >0 \} \leq 2C (\rho \vee 1)^\gamma \one \{ \rho >0 \}
$$
for sufficiently large $n$. Combining these bounds, 
$$
\frac{f\bigl( \Rpn + a(\Rpn)\rho \bigr)}{f(\Rpn )}  \leq 2C (\rho \vee 1)^{\gamma} \sum_{m=0}^\infty \one \bigl\{ 0 < \rho \leq e^{(m+1)\epsilon} / \epsilon \bigr\}\, e^{-m}. 
$$

We next deal with the product terms of the probability densities in \eqref{e:long.expression}. For each $i=1,\dots,p-1$, it follows from \eqref{e:Taylor1}, \eqref{e:uniform.bound}, and \eqref{e:conv.xi} that 
\begin{align*}
\frac{f\bigl( \| (\Rpn + a(\Rpn)\rho)\ta + M y_i \| \bigr)}{f(\Rpn )} &= \frac{L\bigl( \Rpn + a(\Rpn) (\rho + \xi_n (\rho, \ta, y_i)) \bigr)}{L(\Rpn)} \\
&\times \exp \Bigl\{ -\int_0^{\rho + \xi_n(\rho, \ta, y_i)}\frac{a(\Rpn)}{a(\Rpn  + a(\Rpn)r)}\, dr  \Bigr\} \\
&\to e^{-\rho - c^{-1} \langle \ta, y_i \rangle}
\end{align*}
for all $\rho >0$, $\ta \in S^{d-1}$, and $y_i \in \bbr^d$. As for the suitable integrable bound, simply dropping the exponential term, we have 
$$
\frac{f\bigl( \| (\Rpn + a(\Rpn)\rho)\ta + M y_i \| \bigr)}{f(\Rpn )} \leq C^\prime (\rho \vee 1)^{\gamma (p-1)}
$$
for some constant $C^\prime > 0$. 
For the exponential term in \eqref{e:long.expression}, 
\begin{align*}
n &\Q_{2M} \bigl( ( \Rpn +a(\Rpn)\rho )\ta,  (\Rpn + a(\Rpn)\rho)\ta + M \by \bigr) \\
&= nM^d f(\Rpn ) \int_{\B_2(0,\by)} \frac{f\bigl( \| (\Rpn + a(\Rpn)\rho)\ta + M v  \|  \bigr)}{f(\Rpn )}\, dv \to 0,
\end{align*}
since \eqref{e:asym.equ.exponential} implies $nM^d f(\Rpn )  \to 0$, $n\to\infty$. 

Combining all convergence results together, while assuming the applicability of the dominated convergence theorem, we get \eqref{e:first.limit.exp} as desired. Finally, apply all the bounds derived thus far, and note that 
\begin{align*}
\int_0^\infty \sum_{m=0}^\infty \one \bigl\{ 0 < \rho \leq e^{(m+1)\epsilon} / \epsilon \bigr\}\, e^{-m} (\rho \vee 1)^{d+\gamma p -1} d\rho  \leq \left( \frac{e^{\epsilon}}{\epsilon} \right)^{d+\gamma p} \sum_{m=0}^\infty e^{-[1-\epsilon (d+\gamma p)]m} < \infty, 
\end{align*}
since $0 < \epsilon < (d+\gamma p)^{-1}$. Therefore, the dominated convergence theorem is applicable. 
\end{proof}

\bibliography{Takashi_ref}

\end{document}